\documentclass[
 a4paper,
 titlepage=false, 
 abstract=true,
 final, 
toc=bibliography, 
]{scrartcl}

\usepackage[USenglish]{babel}
\usepackage[T1]{fontenc}
\usepackage[utf8]{inputenc}

\usepackage{scrtime,scrdate}
\usepackage{ifthen}
\usepackage{ifdraft}   
\usepackage{amsmath}
\usepackage{amssymb}
\usepackage{bbm}

\usepackage{mathtools}
\usepackage{upgreek}
\usepackage{xspace}
\usepackage{color}
\usepackage{icomma} 
\usepackage{nicefrac}
\usepackage{booktabs}
\usepackage{longtable}
\usepackage{enumitem}
\usepackage{braket}
\usepackage[noadjust]{cite}
\usepackage[
obeyDraft 
, textsize=tiny
, backgroundcolor=yellow
]{todonotes}
\usepackage[
margin=10pt, font=small, labelfont=bf, labelsep=endash, format=plain
]{caption}

\usepackage{pgf,tikz}
\usetikzlibrary{arrows, calc} 
\usepackage{mathrsfs}

\usepackage[
draft=false
]{hyperref}
\hypersetup{
   colorlinks=true,
   citecolor=magenta 
}

\usepackage[amsmath,thmmarks]{ntheorem}

\theoremstyle{break}
\theoremseparator{}
\newtheorem{theorem}{Theorem}[section]

\newtheorem{proposition}[theorem]{Proposition}
\newtheorem{lemma}[theorem]{Lemma}

\newtheorem{corollary}[theorem]{Corollary}
\newtheorem{definition}[theorem]{Definition}
\theorembodyfont{\upshape}

\newtheorem{remark}[theorem]{Remark}

\newtheorem{remarks}[theorem]{Remarks}

\theoremstyle{nonumberplain}
\theoremheaderfont{\itshape}
\theorembodyfont{\upshape}
\theoremseparator{:}
\theoremsymbol{\ensuremath{\Box}}

\newtheorem{proof}{Proof} 

\renewcommand*{\phi}{\varphi}
\renewcommand*{\epsilon}{\varepsilon}

\newcommand*{\ie}{\mbox{i.\,e.}\xspace}

\newcommand*{\eg}{\mbox{e.\,g.}\xspace}

\newcommand*{\cf}{cf.\xspace}

\newcommand*{\wrt}{\mbox{w.\,r.\,t.}\xspace}

\newcommand*{\iid}{\mbox{i.\,i.\,d.}\xspace}

\newcommand*{\A}{\mathbb{A}}

\newcommand*{\E}{\mathbb{E}}
\newcommand*{\M}{\mathbb{M}}
\newcommand*{\N}{\mathbb{N}}
\renewcommand*{\P}{\mathbb{P}}
\newcommand*{\Q}{\mathbb{Q}}
\newcommand*{\R}{\mathbb{R}}
\renewcommand*{\S}{\mathbb{S}}

\newcommand*{\Rplus}{\ensuremath{\mathbb{R}_{\scriptscriptstyle +}}}

\newcommand*{\Qplus}{\ensuremath{{\mathbb{Q}_{\scriptscriptstyle +}}}}
\newcommand*{\Eins}{\mathbbm{1}}
\newcommand*{\One}{\Eins}

\newcommand*{\calA}{\ensuremath{\mathcal{A}}}
\newcommand*{\calB}{\ensuremath{\mathcal{B}}}
\newcommand*{\calC}{\ensuremath{\mathcal{C}}}

\newcommand*{\calE}{\ensuremath{\mathcal{E}}}
\newcommand*{\calF}{\ensuremath{\mathcal{F}}}

\newcommand*{\calK}{\ensuremath{\mathcal{K}}}

\newcommand*{\calM}{\ensuremath{\mathcal{M}}}
\newcommand*{\calN}{\ensuremath{\mathcal{N}}}

\newcommand*{\calP}{\ensuremath{\mathcal{P}}}

\newcommand*{\calR}{\ensuremath{\mathcal{R}}}

\newcommand*{\calT}{\ensuremath{\mathcal{T}}}

\newcommand*{\calX}{\ensuremath{\mathcal{X}}}
\newcommand*{\calY}{\ensuremath{\mathcal{Y}}}

\DeclareMathOperator{\diam}{diam}

\DeclareMathOperator{\I}{I}
\DeclareMathOperator{\id}{id}

\DeclareMathOperator{\supp}{supp}
\DeclareMathOperator*{\wlim}{w-lim}
\DeclareMathOperator{\Complement}{\raisebox{-0.2ex}{$\complement$}}

\newcommand*{\wto}{\xrightarrow{w}}

\newcommand*{\abs}[1]{\ensuremath{\lvert#1\rvert}}

\newcommand*{\norm}[2][]{ \ensuremath{ \ifthenelse{\equal{#1}{}}{\lVert
      #2 \rVert}{\lVert #2 \rVert_{#1}} } }

\newcommand*{\wildcard}{{}\cdot{}}

\newcommand*{\from}{\colon}

\newcommand*{\difop}[1]{\mathop{}\!\mathrm{d}#1}
\newcommand*{\dif}{\difop}

\newcommand*{\assign}{\ensuremath{\mathrel{\mathop:}=}}
\newcommand*{\asn}{\ensuremath{\mathrel{\mathop:}=}}

\newcommand{\rest}[2]{#1\kern.15em\raisebox{-.2em}{$\big|$ \kern
    -.2em #2}}

\newcommand*{\clBall}{\ensuremath{\overline{B}}}

\newcommand*{\ExpDistribution}[1]{\ensuremath{\mathrm{Exp}(#1)}}

\newcommand*{\nullmeasure}{\ensuremath{o}}

\newcommand*{\tfset}{\calT}

\newcommand*{\mass}[1]{
  \ensuremath{
    \ifthenelse{ \equal{#1}{} }
    {\mathfrak{m}}
    {\mathfrak{m}(#1)}
  }
}

\newcommand*{\DD}[1]{w(#1)}

\newcommand*{\MMD}[2][]{
  \ensuremath{
    \ifthenelse{ \equal{#1}{} }
     { \ifthenelse{ \equal{#2}{} }
     {{V_\delta}}
     {{V_\delta(#2)}}
    }
    {
    \ifthenelse{ \equal{#2}{} }
     {{V_{#1}}}
     {{V_{#1}(#2)}}
     }
  }
}
\newcommand*{\AMMD}[2][]{
  \ensuremath{
    \ifthenelse{ \equal{#1}{} }
     { \ifthenelse{ \equal{#2}{} }
     {{U_\delta}}
     {{U_\delta(#2)}}
    }
    {
    \ifthenelse{ \equal{#2}{} }
     {{U_{#1}}}
     {{U_{#1}(#2)}}
     }
  }
}

\newcommand*{\vect}[1]{\mathbf{\boldsymbol{#1}}}
\newcommand*{\matr}[1]{\mathbf{\boldsymbol{#1}}}

\newcommand*{\normalize}[1]{\overline{#1}}

\newcommand*{\mtwom}{\mathbb{M}^{(2)}}
\newcommand*{\Mtwo}{\mtwom}
\newcommand*{\MtwoProb}{\mathbb{M}^{(2)}_{1,1} }   

\newcommand*{\dP}{d_\mathrm{P}}

\newcommand*{\tGwto}{\ensuremath{\xrightarrow{2Gw}}}
\newcommand*{\dtGP}{d_\mathrm{2GP}}

\newcommand*{\tautGP}{\ensuremath{\tau_{2GP}}}
\newcommand*{\tautGw}{\ensuremath{\tau_{2Gw}}}

\newcommand*{\mm}[2][]{
  \ensuremath{
    \ifthenelse{ \equal{#1}{} }
    {M_{#2}}
    {M_{#2, #1}}
  }
}
\newcommand*{\dmm}[2][]{ 
  \ensuremath{
    \ifthenelse{ \equal{#1}{} }
      {\Lambda_{#2}}
      {\Lambda^{#1}_{#2} }
  }
}
\renewcommand*{\dmm}[2]{
  \ensuremath{
  \Lambda^{#1}_{#2}
  }
}
\newcommand*{\nmm}[2][]{
  \ensuremath{
    \ifthenelse{ \equal{#1}{} }
    {\hat{M}_{#2}}
    {\hat{M}_{#2, #1}}
  }
}

\newcommand{\freq}[3]{ 
   \ensuremath{f_{#1, #2}(#3)} 
} 
\newcommand*{\freqvect}[2]{
   \ensuremath{\vect{f}_{#1}(#2)}   
}

\setlist[enumerate, 1]{label=(\arabic*), ref=(\arabic*)}
\setlist[enumerate, 2]{label=(\alph*), ref=\theenumi\alph*}

\begin{document}

\begin{titlepage}
  \title{Convergence of metric two-level measure spaces}
  \author{Roland Meizis\thanks{\enspace Faculty of Mathematics,
      University of Duisburg-Essen, 45141 Essen, Germany \newline
      Email:
      \href{mailto:roland.meizis@uni-due.de}{roland.meizis@uni-due.de}
    }} \date{November 16, 2019}
\end{titlepage}
\maketitle

\begin{abstract}
  We extend the notion of metric measure spaces to so-called metric
  two-level measure spaces (m2m spaces): An m2m space $(X, r, \nu)$ is
  a Polish metric space $(X, r)$ equipped with a two-level measure
  $\nu \in \mathcal{M}_f(\mathcal{M}_f(X))$, i.e. a finite measure on
  the set of finite measures on $X$. We introduce a topology on the
  set of (equivalence classes of) m2m spaces induced by certain test
  functions (\ie the initial topology with respect to these test
  functions) and show that this topology is Polish by providing a
  complete metric.

  The framework introduced in this article is motivated by possible
  applications in biology. It is well suited for modeling the random
  evolution of the genealogy of a population in a hierarchical system
  with two levels, for example, host--parasite systems or populations
  which are divided into colonies. As an example we apply our theory
  to construct a random m2m space modeling the genealogy of a nested
  Kingman coalescent.
\end{abstract}

\paragraph{Keywords:} metric two-level measure spaces, metric measure
spaces, two-level measures, nested Kingman coalescent measure tree,
two-level Gromov-weak topology

\paragraph{AMS MSC 2010:}  Primary 60B10; Secondary 60B05, 92D10

\setcounter{tocdepth}{2} 
\tableofcontents

\section{Introduction}

A metric measure space, hereinafter abbreviated as mm space, is a
triple $(X, r, \mu)$ where $(X, r)$ is a Polish metric space (\ie a
complete and separable metric space) and $\mu$ is a finite Borel
measure on $X$. Metric measure spaces are central objects in
probability theory. Every random variable on a Polish metric space
$(X, r)$ can be identified with its probability distribution $\mu$ on
$X$. Hence, the random variable is represented by the metric measure
space $(X, r, \mu)$. Therefore, metric measure spaces occur almost
everywhere in probability theory, although most of the time only
implicitly. On the other hand, notions of convergence of mm spaces and
metrics on the set of (equivalence classes of) mm spaces have been of
interest in geometric analysis (\cite{Gromov99, KTSturm06,
  LottVillani09}), problems of optimal transport (\cite[in particular
chapter 27]{VillaniBook}) and mathematical biology (\cite{GPW09}). The
introduction of these notions of convergence has allowed for the study
of mm space-valued stochastic processes. This is of particular
interest in mathematical biology, where such processes are used to
study the evolution of genealogical (or phylogenetic) trees (\cf
\cite{GPW13,DGP12,Gloede12,KliemWinter17,Gufler}). Typically, the
metric space $(X, r)$ incorporates the genealogical tree and $\mu$ is
a uniform sampling measure on the leaves of the tree.

In this article we extend the theory of metric measure spaces, in
particular the definitions and results from \cite{GPW09}. In
\cite{GPW09} the authors study metric \emph{probability} measure
spaces, \ie metric measure spaces $(X, r, \mu)$ where $\mu$ is a
probability measure. They introduce the Gromov-weak topology on the
set $\M_1$ of equivalence classes of metric probability measure
spaces. This topology is defined as an initial topology with respect
to a certain set of test functions on $\M_1$. Roughly speaking,
convergence in the Gromov-weak topology is equivalent to weak
convergence of the distributions of sampled finite subspaces.
The authors also introduce the Gromov-Prokhorov metric and prove that
it induces the Gromov-weak topology. In particular, they show that
$\M_1$ equipped with the Gromov-weak topology is Polish and thus a
suitable state space for stochastic processes.

The results in \cite{GPW09} have been generalized to metric measure
spaces with finite measures (\cite{Gloede12}) and extended to marked
metric measures spaces (see \cite{DGP11} for probability measures and
\cite{KliemWinter17} for finite measures). We also seek to extend the
results in \cite{GPW09} by replacing the measure $\mu$ on the metric
space $(X, r)$ with a two-level measure $\nu \in \calM_f(\calM_f(X))$,
\ie with a finite measure on the set of finite measures on $X$. This
extension is motivated by the study of two-level branching systems in
biology, \eg host--parasite systems, where individuals of the first
level are grouped together to form the second level and both levels
are subject to branching or resampling mechanisms.

Let us give a few examples of models of such two-level systems that
can be found in the mathematical literature: 
\begin{enumerate}
  \item\label{it:Examples2LvlModels_3} Dawson, Hochberg and Wu develop
  a two-level branching process in \cite{DHW90, WuThesis}.
  They consider particles which move in $\R^d$ and are subject to a
  birth-and-death process. Moreover, the particles are grouped into
  so-called superparticles, which are subject to another
  birth-and-death process. The state of this process is given by a
  two-level measure $\nu \in \calM(\calM(\R^d))$, \ie a Borel measure
  on the set of Borel measures on $\R^d$.
  The authors also consider the small mass, high density limit of the
  discrete process. This leads to a two-level diffusion process.
  \item The authors in \cite{BansayeTran11} provide a continuous-time
  two-level branching model for parasites in cells. The parasites live
  and reproduce (\ie branch) inside of cells which are subject to cell
  division. At division of a cell the parasites inside are distributed
  randomly between the two daughter cells. The model originates from
  the discrete processes in \cite{Kimmel97, Bansaye08} and
  can be seen as a diffusion limit of these processes.
  \item Another example of a two-level process is given in
  \cite{MeleardRoelly}. The authors model the evolution of a
  population together with different kinds of cells that proliferate
  inside of the individuals of the population. The individuals follow
  a birth-and-death mechanism (including mutation and selection) and
  the cells inside the individuals follow another birth-and-death
  mechanism.
  \item\label{it:Examples2LvlModels_4} Dawson studies two-level
  resampling models in \cite{Dawson17}. He considers a random process
  that models a population which is divided into colonies. The type
  space $X$ of the individuals is finite and the state of the process
  is given by a two-level probability measure
  $\nu \in \calM_1(\calM_1(X))$, \ie a Borel probability measure on
  the set of Borel probability measures on $X$. The individuals are
  subject to mutation, selection, resampling and migration mechanisms
  while at the same time the colonies are also subject to selection
  and resampling mechanisms. The author considers the finite case in
  which the number of colonies and the number of individuals per
  colony are fixed as well as the limit case when both of these
  numbers go to infinity. The limit is a two-level diffusion process
  called the two-level Fleming-Viot process.
\end{enumerate}

We want to extend the theory of metric measure spaces in such a way
that the new framework is suitable for modeling the aforementioned
examples. The hierarchical two-level structure is of particular
interest for us. That is why we study triples $(X, r, \nu)$ where
$(X, r)$ is a Polish metric space and $\nu \in \calM_f(\calM_f(X))$.
We call such a triple a \emph{metric two-level measure space}
(abbreviated as \emph{m2m space}).

Let us give an example how we intend to use an m2m space $(X, r, \nu)$
to model a host--parasite system: The metric space $(X, r)$ represents
the set of parasites $X$ together with its genealogical (or
phylogenetic) tree, which is encoded in the metric $r$. The two-level
measure $\nu$ encodes the distribution of individuals among the hosts.
For example, a single host with two parasites might be represented by
a measure $\delta_{\delta_{x}+\delta_{y}}$, whereas two hosts with one
parasite each might be represented by a measure
$\delta_{\delta_x}+\delta_{\delta_y}$. One can easily imagine more
complicated situations with more hosts and parasites. The theory
developed in this article also allows to model two-level diffusions.
They appear as small mass, high density limits of atomic two-level
measures like above (\cf examples \ref{it:Examples2LvlModels_3} and
\ref{it:Examples2LvlModels_4}).

{
  Before we summarize the content of this article, let us briefly
  compare our approach of metric two-level measure spaces to marked
  metric measure spaces, which have been introduced in~\cite{DGP11}
  and further developed in \cite{DGP12, KliemWinter17}. A marked
  metric measure space is a metric measure space with an additional
  mark space. The mark space is used to associate individuals with
  traits or groups. Thus marked metric measure spaces can also be used
  to model host--parasite systems with the mark space representing the
  set of hosts. There are two main differences to our approach:
  \begin{itemize}
    \item The mark space of a marked metric measure space must be
    fixed before implementing a model. Thus the set of hosts must be
    chosen beforehand. In a setting with metric two-level measure
    spaces the number of hosts is variable and can change over time.
    \item Two-level measures are a necessary ingredient for two-level
    diffusions (\cf the work of Dawson et al.~cited in examples
    \ref{it:Examples2LvlModels_3} and \ref{it:Examples2LvlModels_4}).
    Thus marked metric measure spaces cannot be used to model
    two-level diffusions.
  \end{itemize}
}
In the context of our theory we are only interested in the structure
of the (genealogical) trees and not in their labels. To get rid of
labels, we define a notion of equivalence for m2m spaces. The focus of
this equivalence is on the structure of the measure $\nu$ and its
``effective support in $X$''. By effective support in $X$ we mean the
smallest closed subset $C \subset X$ with $\supp\mu \subset C$ for
$\nu$-almost every $\mu \in \calM_f(X)$. This set is equal to the
support of the first moment measure
$\mm{\nu}(\wildcard) \asn \int \mu(\wildcard) \dif\nu(\mu)$.
Roughly speaking, we identify two m2m spaces $(X, r, \nu)$ and
$(Y, d, \lambda)$ if $\lambda$ can be mapped into $\nu$ with a
function that is isometric on the effective supports of the
measures. To be precise, $(X, r, \nu)$ and $(Y, d, \lambda)$ are said
to be equivalent if there is a function $f \from X \to Y$ which is
isometric on $\supp \mm{\nu}$ (the effective support in $X$ of
$\nu$) and measure preserving in the sense that $\nu = f_{**}\lambda$.
Here, $f_{**}\lambda$ denotes the two-level push-forward of $\lambda$.
It is the push-forward of the push-forward operator of $f$. That is,
$f_{**}\lambda = \lambda \circ f_*^{-1}$, where $f_*$ is the function
from $\calM_f(X)$ to $\calM_f(Y)$ that maps a finite Borel measure
$\mu$ to its push-forward $f_*\mu = \mu \circ f^{-1}$.

This notion of equivalence allows us to consider the set $\Mtwo$ of
all equivalence classes of m2m spaces. We define a set $\tfset$ of
bounded test functions on $\Mtwo$ that separates points in $\Mtwo$.
The set $\tfset$ consists of three types of test functions
$\Phi \from \Mtwo \to \R$, which serve different purposes for
determining an m2m space $(X, r , \nu)$. The first kind is of the form
\begin{align}
      \Phi((X, r, \nu)) &= \chi(\mass{\nu}) \tag{\ref*{TF1}}
\end{align}
where $\chi \in \calC_b(\Rplus)$ with $\chi(0) = 0$ and where
$\mass{\nu}$ denotes the total mass of a measure $\nu$. Test functions
of this form determine the mass $\mass{\nu}$ of the evaluated m2m
space. The second kind is of the form
\begin{align}
    \Phi((X, r, \nu)) &= \chi(\mass{\nu}) \int \psi(\mass{\vect{\mu}})
                        \dif \normalize{\nu}^{\otimes m}(\vect{\mu}) \tag{\ref*{TF2}},
\end{align}
where $m \in \N$ and $\psi \in \calC_b(\Rplus^m)$ with
$\psi(\vect{a}) = 0$ whenever any of the components of the vector
$\vect{a} \in \Rplus^m$ is 0 and where
$\normalize{\nu} = \frac{\nu}{\mass{\nu}}$ denotes the normalization
of $\nu$. The test functions of the form \eqref{TF2} determine the
normalized mass distribution
$\mass{}_*\normalize{\nu} \in \calM_1(\Rplus)$ of the evaluated m2m
space. The third kind of test function is of the form
\begin{align}
      \Phi((X, r, \nu)) &= \chi(\mass{\nu}) \int \psi(\mass{\vect{\mu}})
    \int \phi \circ R(\matr{x}) \dif \normalize{\vect{\mu}}^{\otimes
      \vect{n}}(\matr{x}) \dif \normalize{\nu}^{\otimes m}(\vect{\mu}),
                        \tag{\ref*{TF3}}
\end{align}
where $\vect{n} = (n_1, \dotsc, n_m) \in \N^m$,
$\phi \in \calC_b(\Rplus^{\abs{\vect{n}} \times \abs{\vect{n}}})$ and
$\normalize{\vect{\mu}}^{\otimes \vect{n}}\asn \bigotimes_{i=1}^m
\normalize{\mu_i}^{\otimes n_i}$. Test functions of this form
determine the space $(X, r)$ (more precisely, the support
$\supp\mm{\nu}$ equipped with the restriction of $r$) and the
structure of $\normalize{\nu}$.

The test functions in $\tfset$ are used to induce a Hausdorff topology
on $\Mtwo$. The two-level Gromov-weak topology $\tautGw$ is defined as
the initial topology with respect to $\tfset$, \ie the coarsest
topology on $\Mtwo$ such that all functions in $\tfset$ are
continuous. $\Mtwo$ equipped with the two-level Gromov-weak topology
is in fact a Polish space. To show this, we introduce the two-level
Gromov-Prokhorov metric $\dtGP$, which is complete and metrizes the
two-level Gromov-weak topology. Heuristically, to compute the
two-level Gromov-Prokhorov distance between two m2m spaces
$(X, r, \nu)$ and $(Y, d, \lambda)$ we embed $X$ and $Y$ isometrically
into some common Polish metric space $Z$ and compute the Prokhorov
distance between the two-level push-forwards of $\nu$ and $\lambda$.
The two-level Gromov-Prokhorov distance is defined as the infimum of
this value over all such embeddings.

It turns out that the functions from $\tfset$ are convergence
determining for $\calM_1(\Mtwo)$. Thus, they are a suitable domain for
generators of Markov processes on $\Mtwo$. In particular, this will
allow us to create m2m space-valued analogs of the two-level examples
given above in future research articles.

At the end of our article we apply our framework to an example. We
define a two-level coalescent process called the nested Kingman
coalescent, which is a coalescent model for individuals of different
species. We equip the genealogical tree stemming from this coalescent
with a two-level measure that contains the two-level structure of the
coalescent. The result is a random m2m space called the nested Kingman
coalescent measure tree.

Finally, let us summarize the two main obstacles which arise when we
extend the theory of mm spaces to m2m spaces:
\begin{enumerate}
  \item In the one-level case the set of mm spaces can be embedded
  isomorphically into $\calM_f(\Rplus^{\N \times \N})$ using distance
  matrix distributions (\cf \cite{GPW09}). It follows directly that
  the Gromov-weak topology is metrizable and thus working with
  sequences is sufficient. Such an embedding is not possible anymore
  for m2m spaces. Therefore, it is not a priori clear whether the
  two-level Gromov weak topology is first countable and our proofs
  (for continuity, compactness, etc.) must not rely on sequences.
  Instead, we will work with nets. Nets are a generalization of
  sequences and most of the theorems for metric spaces using sequences
  (\eg continuity of functions, closedness of sets, compactness of
  sets) hold true for general topological spaces when sequences are
  replaced by nets. After proving that $\Mtwo$ equipped with the
  two-level Gromov-weak topology $\tautGw$ is in fact Polish, we then
  go back using ordinary sequences.
  
  \item For characterizing compact sets of m2m spaces it is necessary
  to work with finite first moment measures. Unfortunately, the first
  moment measure
  $\mm{\nu}(\wildcard) = \int \mu(\wildcard) \dif\nu(\mu)$ of a
  two-level measure $\nu \in \calM_f(\calM_f(X))$ may be infinite. We
  can overcome this problem by approximating $\nu$ sufficiently close
  by a two-level measure with finite moment measures. This is done by
  restricting $\nu$ to
  $\calM_{\leq K}(X) = \set{\mu \in \calM_f(X) | \mu(X) \leq K}$ using
  a smooth density function $f_K$. Then, $f_K\cdot \nu$ is an element
  of $\calM_f(\calM_{\leq K}(X))$ and has finite moment measures.
  Moreover, $f_K \cdot \nu$ converges weakly to $\nu$ for
  $K \nearrow \infty$.
\end{enumerate}

\paragraph{Outline:}
\textit{The rest of this article is organized as follows: We start
  with some preliminaries about finite measures, two-level measures
  and nets in Section~\ref{sec:preliminaries}. In the subsequent
  section
  we introduce the notion of metric two-level measure spaces (m2m
  spaces) and the two-level Gromov-weak topology $\tautGw$ on the set
  $\Mtwo$ of (equivalence classes of) m2m spaces. In
  Section~\ref{sec:d2GP-metric} we define the two-level
  Gromov-Prokhorov metric $\dtGP$ on $\Mtwo$ and show that
  $(\Mtwo, \dtGP)$ is a Polish metric space (\ie separable and
  complete). Sections \ref{sec:DDandMMD}, \ref{sec:Approximation} and
  \ref{sec:compactness} are devoted to compact sets in $\Mtwo$. In
  Section~\ref{sec:DDandMMD} we define the distance distribution and
  the modulus of mass distribution for finite measures. In
  Section~\ref{sec:Approximation} we introduce an approximation for
  two-level measures. The approximating two-level measures always have
  finite moment measures. This enables us to characterize compactness
  in $\Mtwo$ in Section~\ref{sec:compactness}. There, we give some
  equivalent conditions for compactness in terms of the distance
  distribution and the modulus of mass distribution. With these
  compactness conditions we are able to prove in
  Section~\ref{sec:EquivalenceOfBothTopologies} that the topology
  induced by the metric $\dtGP$ coincides with the two-level
  Gromov-weak topology $\tautGw$. In Section~\ref{sec:tightness} we
  investigate tightness for probability measures on $\Mtwo$. The
  results about tightness are used in Section~\ref{sec:NestedKingman},
  in which we construct a random m2m space called the nested Kingman
  coalescent measure tree.}

\section{Preliminaries}
\label{sec:preliminaries}

In this section we first summarize some basic properties about the
weak topology and the Prokhorov metric for finite Borel measures. Then
we introduce the first moment measure and the two-level push-forward.
Both are essential ingredients for the definitions of m2m spaces.
Finally, in the last subsection, we introduce nets in topological
spaces.

Throughout the preliminaries and the rest of the article $(X, r)$ will
always be a non-empty Polish metric space, unless otherwise mentioned.
A Polish metric space is a complete and sep\-ar\-able metric space,
while a Polish space is a topological space which is separable and
metrizable with a complete metric.

\subsection{Finite measures and the Prokhorov metric}

By $\calM_f(X)$ we denote the set of all finite Borel measures on $X$,
equipped with the weak topology. The weak topology on $\calM_f(X)$ is
the initial topology with respect to all functions
$\mu \mapsto \int f \dif\mu$ with $f \in \calC_b(X)$, where
$\calC_b(X)$ denotes the set of bounded and continuous functions
from $X$ to $\R$. Recall that the initial topology on a set $A$ with
respect to a set $\calF$ of functions on $A$ is defined as the
coarsest topology on $A$ such that the functions in $\calF$ are
continuous. By definition, a sequence $(\mu_n)_n$ of finite Borel
measures on $X$ converges weakly to a finite Borel measure $\mu$ if
and only if
\begin{displaymath}
  \int f \dif\mu_n \to \int f \dif\mu
\end{displaymath}
for every test function $f \in \calC_b(X)$. 

It is well known that the set $\calM_f(X)$ equipped with the weak
topology is a Polish space and that the Prokhorov metric $\dP$ is a
complete metric for this topology (\cf for example
\cite{Prokhorov56}). The Prokhorov distance $\dP(\mu, \eta)$ between
two finite measures $\mu, \eta \in \calM_f(X)$ is defined as the
infimum over all $\epsilon>0$ such that
\begin{align*}
  \mu(A) \leq \eta(B(A,\epsilon)) + \epsilon \quad \text{and} \quad
           \eta(A) \leq   \mu(B(A,\epsilon)) + \epsilon
\end{align*}
for all closed sets $A \subset X$, where
$B(A, \epsilon) = \bigcup_{a \in A}B(a, \epsilon)$ and
$B(a, \epsilon)$ is the open ball of radius $\epsilon$ around $a$.
To emphasize that we are using the Prokhorov metric for measures on a
specific metric space $(X, r)$, we sometimes write $\dP^X$ or
$\dP^{(X,r)}$ instead of $\dP$.

For $\mu \in \calM_f(X)$ we define the mass of $\mu$ by $\mass{\mu}
\asn \mu(X)$ and the normalization of $\mu$ by
\begin{align*}
  \normalize{\mu} \asn
  \begin{cases}
     \frac{\mu}{\mass{\mu}} &\mu \not= \nullmeasure \\
     \nullmeasure &\mu = \nullmeasure.
  \end{cases}
\end{align*}
Here, $\nullmeasure$ denotes the null measure, which is 0 on all sets.
It is easy to see that the function $\mu \mapsto \normalize{\mu}$ is
continuous on $\calM_f(X)\setminus \set{\nullmeasure}$. For a vector
$\vect{\mu} = (\mu_1, \mu_2, \dotsc, \mu_m) \in \calM_f(X)^m$ we
define $\mass{\vect{\mu}} = (\mass{\mu_1}, \dotsc, \mass{\mu_m})$ and
$\normalize{\vect{\mu}} = (\normalize{\mu_1}, \dotsc,
\normalize{\mu_m})$. Furthermore, for every $K\geq 0$ we define the
sets
\begin{align*}
  \calM_{\leq K}(X) \asn \set{\mu \in \calM_f(X) | \mass{\mu} \leq K} \\
  \intertext{and}
  \calM_K(X) \asn \set{\mu \in \calM_f(X) | \mass{\mu} = K}.
\end{align*}
In particular, $\calM_1(X)$ denotes the set of probability measures on
$X$.

Recall that a set $\calF \subset \calM_f(X)$ is called tight if for
every $\epsilon>0$ there is a compact set $C \subset X$ such that
$\mu(\Complement C) < \epsilon$ for every $\mu \in \calF$. We say that
a single measure $\mu \in \calM_f(X)$ is tight if the set $\set{\mu}$
is tight. Finite measures on Polish spaces are always tight. It is
well known that for \emph{probability} measures tightness is
equivalent to relative compactness. However, for \emph{finite}
measures we also need to ensure that the masses of the measures are
bounded. This is part of the original theorem from Prokhorov
in \cite[Theorem 1.12]{Prokhorov56}.
\begin{proposition}[Prokhorov's Theorem]\label{prp:ProkhorovsTheorem}
  Let $X$ be a Polish space and $\calF \in \calM_f(X)$. $\calF$ is
  relatively compact in the weak topology if and only if $\calF$ is
  tight and the set $\set{ \mass{\mu} | \mu \in \calF}$ is bounded in
  $\R$.
\end{proposition}
Observe that $\set{ \mass{\mu} | \mu \in \calF}$ is bounded if and only
if $\calF$ is bounded in the Prokhorov metric since
\begin{displaymath}
\abs{\mass{\mu}-\mass{\eta}} \leq \dP(\mu, \eta) \leq
\max(\mass{\mu}, \mass{\eta})
\end{displaymath}
for all $\mu, \eta \in \calM_f(X)$.

\subsection{The first moment measure $\mm{\nu}$ and its support}

In this article we deal with \emph{two-level measures} of the form
$\nu \in \calM_f(\calM_f(X))$. They are closely related to random
measures, which are represented by measures in $\calM_1(\calM_f(X))$.
An important tool in the analysis of two-level measures is the
\emph{first moment measure}, also called the \emph{intensity measure}.
The first moment measure of $\nu$ is the Borel measure on $X$ defined
by
\begin{displaymath}
  \mm{\nu}(\wildcard) \asn \int \mu(\wildcard) \dif\nu(\mu).
\end{displaymath}
Note that the first moment measure may be an infinite measure.

\begin{lemma}\label{lm:MomentMeasureIsSupportingMeasure}
  Let $X$ be a Polish space and $\nu \in \calM_f(\calM_f(X))$. The
  first moment measure $\mm{\nu}$ is a \emph{supporting measure} of
  $\nu$ in the sense that
  \begin{displaymath}
    \int f \dif\mm{\nu} = 0 \Leftrightarrow \int f \dif\mu = 0 \text{ for
      $\nu$-almost every $\mu \in \calM_f(X)$}
  \end{displaymath}
  for any non-negative measurable function $f \from X \to \R$.
\end{lemma}
\begin{proof}
  If $0 < \int f \dif\mm{\nu} = \int \int f \dif\mu \dif\nu(\mu)$,
  then the set of all $\mu \in \calM_f(X)$ with $\int f \dif\mu > 0$
  cannot have $\nu$-measure zero. On the other hand, if
  $0 = \int f \dif\mm{\nu}$, then the set of all $\mu \in \calM_f(X)$
  with $\int f \dif\mu > \frac{1}{n}$ must have $\nu$-measure zero for
  every $n \in \N$. Consequently, $\int f \dif\mu = 0$ for $\nu$-almost
  every $\mu \in \calM_f(X)$.
\end{proof}

Recall that the support $\supp \mu$ of a Borel measure $\mu$ is
defined as the smallest closed subset $A$ of $X$ with
$\mu(\Complement A) = 0$. Equivalently, it is the set of all $x \in X$
with $\mu(B(x, \epsilon)) > 0$ for every $\epsilon>0$.
\begin{corollary}\label{cor:SupportOfMomentMeasure}
  Let $X$ be a Polish space and $\nu \in \calM_f(\calM_f(X))$. Then
  $\supp \mu \subset \supp \mm{\nu}$ for $\nu$-almost every
  $\mu \in \calM_f(X)$.
\end{corollary} 
\begin{proof}
  Use Lemma~\ref{lm:MomentMeasureIsSupportingMeasure} with $f \asn
  \One_{\left(\Complement \supp \mm{\nu}\right)}$.
\end{proof}

The previous corollary shows that the two-level measure $\nu$ is
effectively a finite measure on $\calM_f(\supp\mm{\nu})$, \ie we can
restrict $X$ to $\supp\mm{\nu}$ without losing information about $\nu$
(\cf Definition~\ref{def:m2mSpacesAndEquivalence} and
Remark~\ref{rm:Xisometrictosuppmm} for a precise statement).

\subsection{Push-forward operators}

Let $(X, r)$ and $(Y, d)$ be Polish metric spaces and $g$ be a Borel
measurable function from $X$ to $Y$. As usual, $g_*\mu$ denotes the
push-forward measure $\mu \circ g^{-1}$ for a finite Borel measure
$\mu \in \calM_f(X)$. We regard $g_*$ as an operator
\begin{equation}\label{eq:1lvlPushForward}
  \begin{aligned}
    g_* \from \calM_f(X) &\to \calM_f(Y) \\
    \mu &\mapsto g_*\mu = \mu \circ g^{-1}
  \end{aligned}
\end{equation}
and call $g_*$ the \emph{(one-level) push-forward operator of $g$}. This enables
us to define the \emph{two-level push-forward operator $g_{**}$ of
  $g$} by
\begin{equation}\label{eq:2lvlPushForward}
  \begin{aligned}
    g_{**} \colon \calM_f(\calM_f(X)) &\to \calM_f(\calM_f(Y)) \\
    \nu &\mapsto g_{**}\nu \assign \nu \circ (g_*)^{-1}.
  \end{aligned}
\end{equation}
In this article the function $g$ will usually be an isometry between
$X$ and $Y$. Then, the structure of the push-forward measure $g_*\mu$
is the same as of the original measure $\mu \in \calM_f(X)$. The same
is true for the two-level push-forward measure $g_{**}\nu$ with
$\nu \in \calM_f(\calM_f(X))$.

Let $\phi \colon Y \to \R$ be measurable and let $\mu \in \calM_f(X)$
and $\nu \in \calM_f(\calM_f(X))$. The following transformation
formulas hold true for the push-forward measures $g_*\mu$ and
$g_{**}\nu$ (assuming that the integrals exist):
\begin{equation}\label{eq:1lvlTrafo}
  \int \phi \dif(g_*\mu) = \int \phi \circ g \dif\mu 
\end{equation}
and
\begin{equation}\label{eq:2lvlTrafo}
  \begin{aligned}
    \int_{\calM_f(Y)} \int_{Y} \phi \dif\mu \dif(g_{**}\nu)(\mu)
    &= \int_{\calM_f(X)} \int_Y \phi \dif(g_*\mu) \dif\nu(\mu) \\
    &= \int_{\calM_f(X)} \int_X \phi \circ g \dif\mu \dif\nu(\mu).
  \end{aligned}
\end{equation}

The following lemma summarizes some useful properties of the one-level
and two-level push-forward operator.
\begin{lemma}[Properties of push-forward
  operators]\label{lm:pushforwardproperties}
  Let $(X, d_X), (Y, d_Y)$ be Polish metric spaces and
  $h, g, g_1, g_2, \dotsc$ be measurable functions from $X$ to $Y$.
  Then we have:
  \begin{enumerate}
    \item\label{it:lm:pushforwardproperties_1} If $g$ is continuous,
    then $g_*$ and $g_{**}$ defined as in \eqref{eq:1lvlPushForward}
    and \eqref{eq:2lvlPushForward}, respectively, are continuous.
    \item\label{it:lm:pushforwardproperties_2} If $g_n$ converges
    pointwise to $g$, then $g_{n*}$ and $g_{n**}$ converge pointwise
    to $g_*$ and $g_{**}$, respectively.
    That is, $g_{n*}\mu$ converges weakly to $g_*\mu$ and $g_{n**}\nu$
    converges weakly to $g_{**}\nu$ for every $\mu \in \calM_f(X)$ and
    $\nu \in \calM_f(\calM_f(X))$.
    \item\label{it:lm:pushforwardproperties5} Let
    $\mu \in \calM_f(X)$ and $\epsilon>0$. Define
    $M_\epsilon \asn \set{ x \in X | d_Y( g(x), h(x) ) < \epsilon}$
    and $\delta \asn \mu(\Complement M_\epsilon)$, then we have
    \begin{displaymath}
      \dP( g_*\mu , h_*\mu ) \leq \max(\epsilon, \delta).
    \end{displaymath}
  \end{enumerate}
\end{lemma}
\begin{proof}
  {
    The proof of~\ref{it:lm:pushforwardproperties_1} is
    straightforward with the transformation
    formula~\eqref{eq:1lvlTrafo}: Let $(\mu_n)_n$ converge weakly to
    $\mu$ in $\calM_f(X)$ and let $f \in \calC_b(Y)$ be arbitrary.
    Because $f \circ g$ is continuous and bounded, we get
    \begin{displaymath}
      \int f \dif g_*\mu_n = \int f \circ g \dif \mu_n \to \int
      f \circ g \dif \mu = \int f \dif g_*\mu.
    \end{displaymath}
    Thus, $g_*\mu_n$ converges weakly to $g_*\mu$ and $g_*$ is
    continuous. It follows that $F \circ (g_*)$ is continuous for
    every $F \in \calC_b(\calM_f(Y))$. Therefore, if $(\nu_n)_n$
    converges to $\nu$ in $\calM_f(\calM_f(X))$ we get
    \begin{displaymath}
      \int F \dif g_{**}\nu_n = \int F \circ (g_*) \dif \nu_n \to \int F
      \circ (g_*) \dif \nu = \int F \dif g_{**}\nu
    \end{displaymath}
    and this proves the continuity of $g_{**}$.
  }
  
  {
    Assertion~\ref{it:lm:pushforwardproperties_2} follows in a similar
    way as assertion~\ref{it:lm:pushforwardproperties_1} using the
    transformation formula~\eqref{eq:1lvlTrafo} and dominated
    convergence. We omit the proof.
  }
    
  To show~\ref{it:lm:pushforwardproperties5}, let $A \subset X$ be a
  closed set and let $m \asn \max(\epsilon, \delta)$. Then,
  \begin{align*}
    g_*\mu(A) &= \mu(g^{-1}(A)) \\
    &= \mu(g^{-1}(A) \cap M_\epsilon) + \mu(g^{-1}(A) \cap \Complement
    M_\epsilon) \\
    &\leq \mu(h^{-1}(B(A, \epsilon))) + \delta \\
    &\leq h_*\mu(B(A, m)) + m
  \end{align*}
  and in the same way we can show that
 \begin{align*}
   h_*\mu(A) \leq g_*\mu(B(A, m)) + m.
  \end{align*}
  This holds for every closed set $A \subset X$ and thus
   \begin{displaymath}
     \dP( g_*\mu , h_*\mu ) \leq m = \max(\epsilon, \delta).
    \end{displaymath}
\end{proof}

\subsection{Nets in topological spaces}
\label{ssec:nets}

This subsection is a short introduction to nets. A more comprehensive
survey can be found in \cite{Kelley}. Nets are a generalization of
sequences and the reader not familiar with this topic may safely skip
this part and think of sequences whenever we use nets.

A non-empty set $\calA$ with a partial order $\preceq$ is called
\emph{directed} if every pair $\alpha_1, \alpha_2 \in \calA$ has a
common successor $\alpha \in \calA$ (\ie $\alpha_1 \preceq \alpha$ and
$\alpha_2 \preceq \alpha$). A map $x$ from a directed set
$(\calA, \preceq)$ to a topological space $(X, \tau)$ is called a
\emph{net in $X$}. Similar to sequences we will denote this map by
$(x_\alpha)_{\alpha \in \calA}$ or $(x_\alpha)_{\alpha}$. Observe that
$(\N, \leq)$ is a directed set and that a sequence $(x_n)_{n \in \N}$
is a net with index set $\N$.

We say that the net $(x_\alpha)_{\alpha}$ is \emph{eventually} in a
set $A \subset X$ if there is an $\alpha_0 \in \calA$ such that
$x_\alpha \in A$ for all $\alpha \succeq \alpha_0$. We say that
$(x_\alpha)_{\alpha}$ is \emph{frequently} in $A$ if every
$\alpha_0 \in \calA$ has a successor $\alpha \succeq \alpha_0$ with
$x_\alpha \in A$. Likewise, we say that a net \emph{eventually (resp.
  frequently)} has a certain property if it eventually (resp.
frequently) takes values in the set of elements of $X$ with this
property.

Let $z$ be an element of $X$. A net $(x_\alpha)_{\alpha}$ is said to
converge to $z$ if for every neighborhood $N \subset X$ of $z$ there
is an $\alpha_0 \in \calA$ such that $x_\alpha \in N$ for all
$\alpha \succeq \alpha_0$ (\ie the net is eventually in $N$). We
denote this convergence by $x_\alpha \to z$.
\begin{lemma}
  Let $X$ and $Y$ be topological spaces and $f \from X \to Y$. The
  function $f$ is continuous if and only if for every convergent net
  $x_\alpha \to z$ in $X$ we have $f(x_\alpha) \to f(z)$.
\end{lemma}

Let $(\calB, \preceq_\calB)$ be another directed set and
$(y_\beta)_\beta$ be another net in $X$. We say that $(y_\beta)_\beta$
is a \emph{subnet} of $(x_\alpha)_\alpha$ if there is a function $T$
from $\calB$ to $\calA$ with $y = x \circ T$ (\ie
$y_\beta = x_{T(\beta)}$ for every $\beta$) and if for every
$\alpha_0 \in \calA$ there is a $\beta_0 \in \calB$ such that
$T(\beta) \succeq \alpha_0$ for every $\beta \succeq_\calB \beta_0$.

Moreover, we call a point $z \in X$ a \emph{cluster point} of
$(x_\alpha)_\alpha$ if for every neighborhood $N \subset X$ of $z$ and
every $\alpha_0 \in \calA$ there is an $\alpha \succeq \alpha_0$ with
$x_\alpha \in N$ (\ie the net is frequently in $N$). It can be shown
that $z$ is a cluster point of $(x_\alpha)_\alpha$ if and only if
there is a subnet converging to $z$.

We call a net \emph{compact} if every subnet has a convergent subnet.
The following lemma is based on \cite[Lemma 2.3]{Topsoe74}.
\begin{lemma}\label{lm:CompactNets}
  Let $C$ be a subset of a regular topological space $X$. The
  following are equivalent:
  \begin{enumerate}
    \item $C$ is relatively compact.
    \item Every net in $C$ has a converging subnet.
    \item Every net in $C$ has a cluster point.
    \item Every net in $C$ is a compact net.
  \end{enumerate}
\end{lemma}

We define the limit superior of a real-valued net $(x_\alpha)_\alpha$
by
\begin{align*}
  \limsup_\alpha x_\alpha = \lim_\alpha \sup_{\alpha'
  \succeq \alpha} x_{\alpha'},
\end{align*}
\ie it is the limit of the supremum of the tails of the net. The limit
superior is the largest cluster point of the net or $\infty$ if there
is no largest cluster point.

Sometimes we will be concerned with measure-valued nets
$(\mu_\alpha)_\alpha \subset \calM_f(\R)$. We call
$(\mu_\alpha)_\alpha$ \emph{tight} if for every $\epsilon>0$ there
is a compact set $C \subset \R$ such that
$\limsup_{\alpha} \mu_\alpha(\Complement C) < \epsilon$ (\ie
$\mu_\alpha(\Complement C) < \epsilon$ eventually). If
$(\mu_\alpha)_\alpha$ is a compact net (\eg convergent), then it is
also tight.

\section{M2m spaces and the two-level Gromov-weak topology}
\label{sec:m2m-spaces}

In this section we give the basic definitions of metric two-level
measure spaces, equivalence of m2m spaces and the set $\Mtwo$ of
equivalence classes of m2m spaces. Moreover, we define a set $\tfset$
of test functions on $\Mtwo$ and show that $\tfset$ separates points
in $\Mtwo$, \ie that an m2m space $\calX \in \Mtwo$ is determined by
the values $\set{ \Phi(\calX) | \Phi \in \tfset}$. Then we define the
two-level Gromov-weak topology on $\Mtwo$ as the initial topology
induced by $\tfset$.

\begin{definition}[Metric two-level measure spaces and
  equivalences]\label{def:m2mSpacesAndEquivalence}
  \begin{enumerate}
    \item A triple $(X, r, \nu)$ is called a \emph{metric two-level
      measure space (m2m space)} if $X \subset \R^\N$ is non-empty,
    $(X, r)$ is a Polish metric space and
    $\nu \in \calM_f(\calM_f(X))$.
    \item Two m2m spaces $(X, r, \nu)$ and $(Y, d, \lambda)$ are
    called equivalent if there exists a measurable function
    $f \from X \to Y$ such that $\lambda = f_{**}\nu$ and $f$ is
    isometric on the set $\supp\mm{\nu}$ (but not necessarily on the
    whole space $X$). The equivalence between both spaces is denoted
    by $(X, r, \nu) \cong (Y, d, \lambda)$ or by
    $(X, r, \nu) \cong_f (Y, d, \lambda)$ if we want to emphasize that
    $f$ is the measure-preserving isometry.
    \item By $\Mtwo$ we denote the set of all equivalence classes of
    m2m spaces. In the following we will not distinguish between an
    m2m space and its equivalence class. Generic elements of $\Mtwo$
    will be denoted by $\calX = (X, r, \nu)$,
    $\calX_n = (X_n, r_n, \nu_n)$ or $\calY = (Y, d, \lambda)$.
  \end{enumerate}
\end{definition}

\begin{remarks}\label{rm:Xisometrictosuppmm}
  \begin{enumerate}
    \item Note that every Polish metric space is homeomorphic to a
    closed subset of $\R^\N$ (equipped with the product topology) by
    \cite[Corollary 4.3.25]{Engelking}. Therefore, the condition
    $X \subset \R^\N$ is not a restriction and every Polish metric
    space with a two-level measure can be seen as an m2m space, even
    if $X$ is not a subset of $\R^\N$.
    \item Let $(X, r, \nu)$ be an m2m space with
    $S \asn \supp\mm{\nu} \not= \emptyset$. The support of $\nu$ is a
    subset of $\set{\mu \in \calM_f(X) | \supp\mu \subset S}$ by
    Corollary~\ref{cor:SupportOfMomentMeasure}. Thus, $(X, r, \nu)$ is
    equivalent to $(S, r', \nu')$, where $r'$ is the restriction of
    $r$ to $S \times S$ and $\nu'$ is the restriction of $\nu$ to
    $\calM_f(S)$. This holds for every m2m space and to simplify our
    proofs we will often assume (without loss of generality) that
    $X = \supp\mm{\nu}$.
  \end{enumerate}
\end{remarks}

Before we start to define the test functions which shall induce the
two-level Gromov-weak topology, we need to introduce some notation:
Let $\vect{\mu} = (\mu_1, \dotsc, \mu_m) \in \calM_f(X)^m$ and
$\vect{n} = (n_1, \dotsc, n_m) \in \N^m$. We define
\begin{align*}
  \vect{\mu}^{\otimes \vect{n}}\asn \bigotimes_{i=1}^m \mu_i^{\otimes
  n_i} \\
  \intertext{and}
  \vect{\mu}^{\otimes\N} \asn \bigotimes_{i=1}^m \mu_i^{\otimes \N}.
\end{align*}
Moreover, for $\vect{\mu} = (\mu_1, \mu_2, \dotsc) \in \calM_f(X)^\N$
we define
\begin{align*}
  \vect{\mu}^{\otimes \N} \asn \bigotimes_{i=1}^\infty \mu_i^{\otimes
  \N}.
\end{align*}
This notation will shorten our test functions and will be particularly
convenient in the upcoming proofs. By a slight abuse of notation, we
sometimes write $(i, j) \in \vect{n}$, where we regard the vector
$\vect{n}$ as the set
$\set{(i,j) | i \in \set{1, \dotsc, m}, j \in \set{1, \dotsc, n_i}}$.
Moreover, for a Polish metric space $(X, r)$, $m \in \N$ and
$\vect{n} = (n_1, \dotsc, n_m) \in \N^m$ we define the following
\emph{distance operators}
\begin{align*}
  R^{(X, r)}_m &\from X^m \to \Rplus^{m \times m} , &  R^{(X, r)}_m(\vect{x}) &\asn
                                                                                (r(x_i, x_j))_{1 \leq i,j \leq m}   \\
  R^{(X, r)}_{\vect{n}} &\from X^{\vect{n}} \to \Rplus^{\abs{\vect{n}}
                          \times \abs{\vect{n}}}, & R^{(X,
                                                    r)}_{\vect{n}}(\vect{x}) &\asn  (r(x_{ij}, x_{kl}))_{(i, j), (k, l) \in 
                                                                               \vect{n}} \\
  R^{(X, r)}_{\N \times \N} &\from X^{\N \times \N} \to \Rplus^{\N^4},
                                                    & R^{(X,
                                                      r)}_{\vect{n}}(\vect{x}) &\asn  (r(x_{ij}, x_{kl}))_{(i, j), (k, l) \in \N^2},
\end{align*}
{where $\abs{\vect{n}} = \sum_{i=1}^m \abs{n_i}$}. For convenience we
often suppress the super- and subscript in the distance operators
above and simply write $R$ instead of $R^{(X, r)}_m$,
$R^{(X, r)}_{\vect{n}}$ and $R^{(X, r)}_{\N \times \N}$. The space and
dimension should always be clear from the context.

\begin{definition}[Test functions]\label{def:TF}
  Define $\tfset$ as the set of all test functions $\Phi \from \Mtwo
  \to \R$ that are of one of the following forms:
  \begin{align}
    \Phi((X, r, \nu)) &= \chi(\mass{\nu}) \tag{TF1},\label{TF1}
    \vphantom{\int} \\
    \Phi((X, r, \nu)) &= \chi(\mass{\nu}) \int \psi(\mass{\vect{\mu}})
    \dif \normalize{\nu}^{\otimes m}(\vect{\mu}) \tag{TF2},\label{TF2}
    \\
    \Phi((X, r, \nu)) &= \chi(\mass{\nu}) \int \psi(\mass{\vect{\mu}})
    \int \phi \circ R(\matr{x}) \dif \normalize{\vect{\mu}}^{\otimes
      \vect{n}}(\matr{x}) \dif \normalize{\nu}^{\otimes m}(\vect{\mu})
    \tag{TF3},\label{TF3}
  \end{align}
  where $m \in \N$, $\vect{n} = (n_1, \dotsc, n_m) \in \N^m$,
  $\chi \in \calC_b(\Rplus)$, $\psi \in \calC_b(\Rplus^m)$,
  $\phi \in \calC_b(\Rplus^{\abs{\vect{n}} \times \abs{\vect{n}}})$
  with $\chi(0) = 0$ and $\psi(\vect{a}) = 0$ whenever any of the
  components of the vector $\vect{a} \in \Rplus^m$ is 0.
\end{definition}

The test functions are created in such a way that they are bounded on
$\Mtwo$. This is the reason why we need to decompose the measures $\nu$
and $\vect{\mu}$ into the masses $\mass{\nu}$, $\mass{\vect{\mu}}$
and the normalized measures $\normalize{\nu}$ and
$\normalize{\vect{\mu}}$. Later in this section we will define a
topology on $\Mtwo$ such that the test functions are continuous. Thus,
they are a suitable domain for generators of stochastic processes on
$\Mtwo$.

In Theorem~\ref{th:ReconstructionTheorem} we will prove that $\tfset$
separates points in $\Mtwo$. The three types of test functions serve
different purposes for determining an m2m space $(X, r, \nu)$. Test
functions of the form \eqref{TF1} simply determine the mass
$\mass{\nu}$, whereas test functions of the form \eqref{TF2} determine
the normalized mass distribution $\mass{}_*\normalize{\nu}$. The space
$(X, r)$ (more precisely, the support $\supp\mm{\nu}$ equipped with
the restriction of $r$) and the structure of $\normalize{\nu}$ are
determined by the test functions of type \eqref{TF3}.

Note that the set $\tfset$ is closed under multiplication, but not
under addition. Moreover, the functions in $\tfset$ are well-defined,
as we can see in the next lemma.
\begin{lemma}\label{lm:TestFunctionsAreWellDefined}
  Every $\Phi \in \tfset$ is well-defined. That is, we have
  $\Phi(\calX) = \Phi(\calY)$ for equivalent m2m spaces
  $\calX, \calY \in \Mtwo$.
\end{lemma}
\begin{proof}
  Let $\calX = (X, r, \nu) \cong_f \calY = (Y, d, \lambda)$. That is,
  $f \from X \to Y$ is isometric on $\supp \mm{\nu}$ and
  $\lambda = f_{**}\nu$. For $n \in \N$ we define
  $f^{\otimes n} \from X^n \to Y^n$ by
  $f^{\otimes n}(x_1, \dotsc, x_n) = (f(x_1), \dotsc, f(x_n))$.
  Because of the isometric properties of $f$, we have
  $R^{(Y, d)}(f^{\otimes n}(\vect{x})) = R^{(X, r)}(\vect{x})$ for
  $\vect{x} \in (\supp\mm{\nu})^n$. Therefore, we conclude that for
  every $\Phi$ as in \eqref{TF3}
  \begin{align*}
    \Phi((Y, d, \lambda)) &= \Phi((Y, d, f_{**}\nu)) \\
    &= \chi( \mass{f_{**}\nu}) \int \psi(\mass{\vect{\mu}}) \int \phi
      \circ R^{(Y, d)}	(\matr x) \dif \vect{\normalize\mu}^{\otimes \vect
      n}(\matr x)\dif
      \normalize{f_{**}\nu}^{ \otimes m}(\vect\mu) \\
    &= \chi( \mass{\nu} ) \int \psi(\mass{{f^{\otimes n}}_*\vect\mu}) \int \phi
      \circ R^{(Y, d)}	(\matr x) \dif ({f^{\otimes m}}_*\vect{\normalize\mu})^{\otimes
      \vect{n}} (\matr x) \dif \normalize{\nu}^{\otimes m}(\vect \mu) \\
    &= \chi( \mass{\nu} ) \int \psi(\mass{\vect\mu}) \int \phi \circ
      R^{(Y, d)}	(f^{\otimes\abs{\vect{n}}}(\matr x)) \dif \vect{\normalize\mu}^{\otimes \vect n}(\matr
      x) \dif \normalize{\nu}^{\otimes m}(\vect \mu) \\
    &= \Phi((X, r,
      \nu)).
  \end{align*}
  Equality for test functions as in \eqref{TF1} and \eqref{TF2}
  follows in the same way.
\end{proof}

Before we are able to prove that $\tfset$ separates points in $\Mtwo$,
we need some preparatory propositions. Let $X$ be a Polish space and
$\mu$ be a Borel probability measure on $X$. For a sequence
$\vect{x} = (x_i)_{i \in \N} \in X^\N$ and $n \in \N$ we define the
empirical measures
\begin{align*}
  \Xi_n(\vect{x}) \asn \frac{1}{n} \sum_{i = 1}^n \delta_{x_i}.
\end{align*}
The Glivenko-Cantelli Theorem (\cf for example \cite[Theorem
11.4.1]{Dudley02}) states that if $\vect{x}$ is random with law
$\mu^{\otimes \N}$, then almost surely the weak limit
$\wlim_{n \to \infty} \Xi_n(\vect x)$ exists and is equal to $\mu$ .
In other words, the measure $\mu$ can be reconstructed from an
infinite \iid sample. This also holds if $\mu$ is a random probability
measure, as can be seen from the following proposition from
\cite[Theorem 11.2.1]{Dawson93}.
\begin{proposition}[Reconstruction of two-level probability
  measures]\label{prp:deFinetti}
  Let $\nu \in \calM_1(\calM_1(X))$ be a two-level probability measure
  on a non-empty Polish space $X$ and
  $\vect x = (x_i)_{i \in \N} \in X^\N$ be a random sequence with law
  $\int \mu^{\otimes\N} \dif\nu(\mu)$. Then the weak limit
  \begin{displaymath}
    \mu \asn \wlim_{n \to \infty} \Xi_n(\vect x)
   \end{displaymath}
   exists almost surely and the random probability measure $\mu$ has
   law $\nu$.
\end{proposition}

Our goal is to reconstruct the two-level measure
$\nu \in \calM_1(\calM_1(X))$ from an infinite sample in $X$. To this
end, we need not one random measure $\mu$, but an \iid sequence of
random measures $(\mu_i)_i$. If we sample an infinite \iid sample
$(x_{ij})_j$ with each measure $\mu_i$, then, from all of the sampled
points we can reconstruct first the measures $\mu_i$ and then the
two-level measure $\nu$. To be precise, let $\matr{x} = (x_{ij})_{ij}$ be
an infinite random matrix in $X$ with law
$\int \vect{\mu}^{\otimes\N}(\wildcard)\dif\nu^{\otimes\N}(\vect{\mu})$. Then,
almost surely for each $i \in \N$ the weak limit
$\mu_i \asn \wlim_{n \to \infty}\Xi_n((x_{ij})_j)$ exists and each
random measure $\mu_i$ has law $\nu$. Therefore,
$\vect{\mu} = (\mu_i)_i$ is an infinite \iid sample in $\calM_1(X)$
and $\nu$ is the weak limit of $\Xi_n(\vect{\mu})$ by the
Glivenko-Cantelli Theorem.

This result can easily be extended to measures
$\nu \in \calM_1(\calM_f(X))$ by decomposing finite measures
$\mu \in \calM_f(X)$ into their total mass $\mass{\mu}$ and their
normalized (probability) measure $\normalize{\mu}$, provided that
$\nu(\set{\nullmeasure}) = 0$ (since we cannot sample points with the null
measure $\nullmeasure$). Let us first state a generalization of
Proposition~\ref{prp:deFinetti}. We omit the straightforward proof.
\begin{lemma}\label{lm:deFinettiExtended}
  Let $X$ be a non-empty Polish space and
  $\nu \in \calM_1(\calM_f(X))$ be a two-level measure with
  $\nu(\set{\nullmeasure}) = 0$. Let
  $(m, \vect x) \in \Rplus \times X^\N$ be random with law
  $\int \delta_{\mass{\mu}} \otimes \normalize\mu^{\otimes \N} \dif
  \nu(\mu)$. Then the weak limit
  \begin{displaymath}
    \mu \asn m \cdot \wlim_{n \to \infty}\Xi_n(\vect x)
  \end{displaymath}
  exists almost surely and the random measure $\mu$ has law
  $\nu$.
\end{lemma}
Thus, we can reconstruct a randomly chosen finite measure $\mu$ from
its mass and a sequence of \iid samples in $X$. It follows that we can
reconstruct a two-level measure $\nu \in \calM_1(\calM_f(X))$ from a
sample of masses and points in $X$:
\begin{proposition}[Reconstruction of two-level
  measures]\label{prp:ReconstructionFromPointsAndMasses}
  Let $X$ be a non-empty Polish space and let
  $\nu \in \calM_1(\calM_f(X))$ be a two-level measure with
  $\nu(\set{\nullmeasure})=0$. Moreover, let
  $( (m_i)_i, (x_{ij})_{ij} ) \in \R^\N \times X^{\N \times \N}$ be
  random with law
  \begin{displaymath}
    \int \bigotimes_{i=1}^\infty \delta_{\mass{\mu_i}} \otimes
    \vect{\normalize{\mu}}^{\otimes \N} \dif\nu^{\otimes \N}(\vect\mu).
\end{displaymath}
  Then almost surely we have
  \begin{enumerate}
    \item\label{it:prp:ReconstructionFromPointsAndMasses1} The weak limit
    $\mu_i \asn m_i \cdot \wlim_{n \to \infty}\Xi_n((x_{ij})_j)$
    exists for every $i \in \N$ and the random measure
    $\mu_i$ has law $\nu$,
    \item\label{it:prp:ReconstructionFromPointsAndMasses2} The
    two-level measure $\frac{1}{n} \sum_{i=1}^n
    \delta_{\mu_i}$ converges weakly to $\nu$,
    \item\label{it:prp:ReconstructionFromPointsAndMasses3} The sequence
    $(x_{ij})_{i}$ is dense in $\supp
    \mm{\nu}$ for every positive integer $j$.
  \end{enumerate}
\end{proposition}
\begin{proof}
  \begin{enumerate}
    \item If we fix $i\in\N$, then $(m_i, (x_{ij})_{j} )$ has law
    $\int \delta_{\mass{\mu}}\otimes \normalize\mu^{\otimes\N} \dif\nu(\mu)$
    and the claim follows from Lemma~\ref{lm:deFinettiExtended}.
    \item $(\mu_i)_i$ is an \iid sequence of random measures, each
    with law $\nu$, and the claim follows from the Glivenko-Cantelli
    Theorem.
    \item $(x_{ij})_i$ is an \iid sequence of random variables in
    $X$, each with law $\int \normalize\mu
    \dif\nu(\mu)$. Thus, the sequence is almost surely dense in its
    support, which coincides with the support of $\mm{\nu}$.
  \end{enumerate}
\end{proof}
Property \ref{it:prp:ReconstructionFromPointsAndMasses3} of the last
proposition implies that we can even reconstruct $\nu$ \emph{and} $(X,
r)$ if we only know the masses
$(m_i)_i$ of the sampled measures and the distances $r(x_{ij},
x_{kl})$ between all points $x_{ij}$ and $x_{kl}$ with $i, j, k, l \in
\N$. We exploit this fact in the next theorem to show that the test
functions in $\tfset$ separate points in $\Mtwo$.
\begin{theorem}[Test functions separate
  points]\label{th:ReconstructionTheorem}
  $\tfset$ separates points in $\Mtwo$. That is, two m2m spaces $\calX,
  \calY \in \Mtwo$ are equivalent if and only if $\Phi(\calX) =
  \Phi(\calY)$ for every test function $\Phi \in \tfset$.
\end{theorem}
\begin{proof}
  We have already shown in Lemma \ref{lm:TestFunctionsAreWellDefined}
  that the value of a test function is the same for equivalent m2m
  spaces. To prove the other direction, let $\calX = (X, r,
  \nu)$ and $\calY = (Y, d, \lambda)$ be such that $\Phi(\calX) =
  \Phi(\calY)$ for every $\Phi \in
  \tfset$. We exclude the trivial case $\nu = \lambda =
  \nullmeasure$. Because $\Phi(\calX) =
  \Phi(\calY)$ for every test function as in \eqref{TF1} and
  \eqref{TF2}, we conclude that $\mass{\nu} =
  \mass{\lambda}$ and $\mass{}_*\nu = \mass{}_*\lambda$.
  In particular we have $\nu(\set{\nullmeasure}) =
  \lambda(\set{\nullmeasure})$. Without loss of generality we may assume
  that $\nu$ and
  $\lambda$ are \emph{probability} measures with $\nu(\set{\nullmeasure}) =
  \lambda(\set{\nullmeasure}) = 0$. Now, from the equality for all $\Phi \in
  \tfset$ it follows that
  \begin{align*}
    \int \bigotimes_{i=1}^m \delta_{\mass{\mu_i}} \otimes
    R^{(X, r)}_*\normalize{\vect\mu}^{\otimes\vect{n}} \dif\nu^{\otimes m}(\vect \mu) =
    \int \bigotimes_{i=1}^m \delta_{\mass{\mu_i}} \otimes
    R^{(Y, d)}_*\normalize{\vect\mu}^{\otimes\vect{n}} \dif\lambda^{\otimes m}(\vect \mu)
  \end{align*}
  for all $m \in \N$ and $\vect{n} \in
  \N^m$. By taking the projective limit we get
  \begin{align*}
    \int \bigotimes_{i=1}^\infty \delta_{\mass{\mu_i}} \otimes
    R^{(X, r)}_*\normalize{\vect\mu}^{\otimes\N} \dif\nu^{\otimes \N}(\vect \mu) =
    \int \bigotimes_{i=1}^\infty \delta_{\mass{\mu_i}} \otimes
    R^{(Y, d)}_*\normalize{\vect\mu}^{\otimes\N} \dif\lambda^{\otimes \N}(\vect \mu).
  \end{align*}
  Together with
  Proposition~\ref{prp:ReconstructionFromPointsAndMasses} this implies
  that there are $\vect m \in \Rplus^\N$, $\matr x \in X^{\N \times
    \N}$ and $\matr y \in Y^{\N \times \N}$ such that $(\vect{m},
  R^{(X, r)}(\matr{x})) = (\vect{m}, R^{(Y,
    d)}(\matr{y}))$ and that both $(\vect{m},
  \matr{x})$ and $(\vect{m},
  \matr{y})$ satisfy the
  properties~\ref{it:prp:ReconstructionFromPointsAndMasses1}
  to~\ref{it:prp:ReconstructionFromPointsAndMasses3} from
  Proposition~\ref{prp:ReconstructionFromPointsAndMasses}. If we
  define a function $f$ by $f(x_{ij}) = y_{ij}$ for all $i,j \in
  \N$, then
  $f$ is isometric and can be extended continuously to an isometric
  function $\bar f \from \supp\mm{\nu} \to
  Y$. Since we can reconstruct $\nu$ from $(\vect m, \matr
  x)$ and $\lambda$ from $(\vect m, \matr
  y)$ by taking weak limits and $\bar{f}_{*}$,
  $\bar{f}_{**}$ are continuous (\cf
  Lemma~\ref{lm:pushforwardproperties}), we see that $\lambda =
  \bar{f}_{**}\nu$.
\end{proof}

\begin{remark}\label{rm:MtwoCannotBeEmbedded}
  In the preceding proof we reconstruct an m2m space $(X, r,
  \nu)$ with $\nu(\set{\nullmeasure})=0$ and $\mass{\nu} =
  1$ from the measure
  \begin{align*}
    \int \bigotimes_{i=1}^\infty \delta_{\mass{\mu_i}} \otimes
    R_*\normalize{\vect\mu}^{\otimes\N} \dif\nu^{\otimes \N}(\vect \mu)
  \end{align*}
  in $\calM_1(\Rplus^\N \times
  \Rplus^{\N^4})$. One might be tempted to think that a general m2m
  space $(X, r, \nu)$ is determined by the measure
  \begin{align}\label{eq:rm:MtwoCannotBeEmbedded}
    \mass{\nu} \cdot \int \bigotimes_{i=1}^\infty \delta_{\mass{\mu_i}} \otimes
    R_*\normalize{\vect{\mu}}^{\otimes\N} \dif\normalize{\nu}^{\otimes
    \N}(\vect{\mu})
  \end{align}
  in $\calM_f(\Rplus^\N \times
  \Rplus^{\N^4})$. However, the association between an m2m space and
  the measure \eqref{eq:rm:MtwoCannotBeEmbedded} is not unique: For
  instance, \eqref{eq:rm:MtwoCannotBeEmbedded} is equal to
  $\nullmeasure$ for every m2m space $(\emptyset, 0,
  c\delta_\nullmeasure)$ with $c \geq 0$, even though they are
  different m2m spaces (\cf Remark~\ref{rm:Xisometrictosuppmm}).
\end{remark}

{
  Since the test functions in $\tfset$ separate points in $\Mtwo$, we
  can use them to induce a topology on $\Mtwo$ which satisfies the
  Hausdorff separation axiom.
  }
\begin{definition}[Two-level Gromov-weak topology]
  The \emph{two-level Gromov-weak topology} is the initial topology on
  $\Mtwo$ induced by the test functions in $\tfset$. That is, the
  two-level Gromov-weak topology is the coarsest topology on $\Mtwo$
  such that the functions in $\tfset$ are continuous. We denote this
  topology by $\tautGw$.

  We say that a net of m2m spaces
  $(\calX_\alpha)_\alpha \subset \Mtwo$ \emph{converges two-level
    Gromov-weakly} to an m2m space $\calX \in \Mtwo$ if it converges
  in the two-level Gromov-weak topology, that is, if
  $\Phi(\calX_\alpha) \to \Phi(\calX)$ for every test function
  $\Phi \in \tfset$. We denote this by $\calX_\alpha \tGwto \calX$.
\end{definition}

\begin{remark}[Weak convergence implies two-level Gromov-weak
  convergence]\label{rm:WeakConvImpliestGwConv}
  Weak convergence of two-level measures implies two-level Gromov-weak
  convergence of the associated m2m spaces. That is, if $(X, r)$ is a
  fixed Polish metric space, then the function
  \begin{equation}\label{eq:rm:WeakConvImpliestGwConv1}
    \begin{aligned}
      \calM_f(\calM_f(X)) &\to \R \\
      \nu &\mapsto \Phi((X, r, \nu))
    \end{aligned}
  \end{equation}  
  is continuous for every $\Phi \in \tfset$ and therefore
  $\nu_n \wto \nu$ implies that $(X, r, \nu_n)_n$ converges two-level
  Gromov-weakly to $(X ,r, \nu)$.

  {
    To see that this is true, let $\Phi$ be a test function of the
    form~\eqref{TF3} (the arguments for the simpler cases~\eqref{TF1}
    and \eqref{TF2} are similar) with $m$, $\vect{n}$, $\chi$, $\psi$,
    $\phi$ as in Definition~\ref{def:TF}. The function
    \begin{equation}\label{eq:rm:WeakConvImpliestGwConv2}
      \begin{aligned}
        \calM_f(X)^m &\to \R \\
        \vect{\mu} &\mapsto \psi(\mass{\vect{\mu}}) \int \phi \circ R
        \dif\normalize{\vect{\mu}}^{\otimes \vect{n}}
      \end{aligned}
    \end{equation}
    is a composition of continuous functions except for the
    normalizing function $\vect{\mu} \mapsto \normalize{\vect{\mu}}$
    which is only discontinuous on
    $\set{\vect{\mu} = (\mu_1, \dotsc, \mu_m) \in \calM_f(X)^m | \mu_i
      = \nullmeasure \text{ for some } i \in \set{1, \dotsc, m}}$.
    However, this discontinuity is smoothed by the fact that
    $\psi(\vect{a})=0$ whenever any of the components of the vector
    $\vect{a}$ is 0 (\cf~Definition~\ref{def:TF}) and therefore
    \eqref{eq:rm:WeakConvImpliestGwConv2} is a continuous function. In
    a similar manner it can be seen that the function
    \begin{equation}\label{eq:rm:WeakConvImpliestGwConv3}
      \begin{aligned}
        \calM_f(\calM_f(X)) &\to \R \\
        \nu &\mapsto \chi(\mass{\nu}) \int F(\vect{\mu}) \dif
        \normalize{\nu}^{\otimes m}(\vect{\mu})
      \end{aligned}
    \end{equation}
    is continuous for every $F \in \calC_b(\calM_f(X)^m)$ (using the
    fact that $\chi(0)=0$). Putting
    \eqref{eq:rm:WeakConvImpliestGwConv3} and
    \eqref{eq:rm:WeakConvImpliestGwConv2} together yields the
    continuity of \eqref{eq:rm:WeakConvImpliestGwConv1}.
  }
\end{remark}

\begin{remark}
  The theory of m2m spaces can be seen as a generalization of the
  theory of metric measure spaces (mm spaces). Let $(X, r, \mu)$,
  $(X_1, r _1, \mu_1)$, $(X_2, r_2, \mu_2), \dotsc$ be mm
  spaces. Then we have
  \begin{enumerate}
    \item $(X_1, r_1, \mu_1)$ and $(X_2, r_2, \mu_2)$ are equivalent
    as mm spaces if and only if $(X_1, r_1, \delta_{\mu_1})$ and
    $(X_2, r_2, \delta_{\mu_2})$ are equivalent as m2m spaces.
    \item $(X_n, r_n, \mu_n)$ converges Gromov-weakly to $(X, r,
    \mu)$ if and only if $(X_n, r_n, \delta_{\mu_n})$ converges
    two-level Gromov-weakly to $(X, r, \delta_{\mu})$
  \end{enumerate}
\end{remark}

It is shown in Section~\ref{sec:EquivalenceOfBothTopologies} that
$(\Mtwo, \tautGw)$ is a Polish space. Thus, subsets are compact if and
only if they are sequentially compact and functions on $\Mtwo$ are
continuous if and only if they are sequentially continuous. But for
now we cannot use the fact that $(\Mtwo, \tautGw)$ is Polish and our
proofs (for compactness or continuity) must not rely on sequences.

Sometimes it is convenient to work with simpler test functions
than~\eqref{TF3}. For this reason we give equivalent conditions for
two-level Gromov-weak convergence in the following lemma.
\begin{lemma}\label{lm:testfunctions_alternative}
  Let $(\calX_\alpha)_{\alpha \in \calA}$ be a net of m2m spaces with
  $\calX_\alpha = (X_\alpha, r_\alpha, \nu_\alpha)$ and let
  $\calX = (X, r, \nu)$ be another m2m space. 
  $(\calX_\alpha)_\alpha$ converges two-level Gromov-weakly to $\calX$
  if and only if the following two conditions hold:
  \begin{enumerate}
    \item\label{it:lm:testfunctions_alternative2_2}
    $\mass{}_*\nu_\alpha$ converges weakly to $\mass{}_*\nu$ in
    $\calM_f(\Rplus)$
    \item\label{it:lm:testfunctions_alternative2_3}
    $\tilde{\Phi}(\calX_\alpha)$ converges to $\tilde{\Phi}(\calX)$
    for every $\tilde{\Phi} \from \Mtwo \to \R$ of the form
    \begin{align}\label{TF4}\tag{TF4}
      \tilde{\Phi}((X, r, \nu)) = \int \psi(\mass{\vect{\mu}}) \int
      \phi \circ R \dif \normalize{\vect{\mu}}^{\otimes\vect{n}}
      \dif \nu^{\otimes m}(\vect{\mu})
    \end{align}
    with $m \in \N$, $\vect{n} \in \N^m$,
    $\phi \in \calC_b(\Rplus^{\abs{\vect{n}} \times \abs{\vect{n}}})$,
    and $\psi \in \calC_b(\Rplus^m)$ with $\psi(\vect{a}) = 0$
    whenever any of the components of the vector $\vect{a} \in \Rplus^m$
    is 0.
  \end{enumerate}
\end{lemma}
Note that in \eqref{TF4} we use $\nu$ instead of the normalized
version $\normalize{\nu}$.
\begin{proof}
  \ref{it:lm:testfunctions_alternative2_2} is equivalent to
  convergence of all test functions of the form \eqref{TF1} and
  \eqref{TF2}. \ref{it:lm:testfunctions_alternative2_3} follows from
  the convergence of functions of the form \eqref{TF3} by choosing the
  function $\chi \in \calC_b(\Rplus)$ such that $\chi(x) = x^m$ for
  $x \in [0, \mass{\nu}+1]$. The other direction is obvious.
\end{proof}

\begin{remark}\label{rm:MassIs2lvlGromovWeakContinuous}
  The previous lemma shows that the functions $(X, r, \nu) \mapsto
  \mass{\nu}$ and $(X, r, \nu) \mapsto \mass{}_*\nu $ are continuous
  on $\Mtwo$ in the two-level Gromov-weak topology.
\end{remark}

\section{The two-level Gromov-Prokhorov metric}
\label{sec:d2GP-metric}

In this section we define the two-level Gromov-Prokhorov distance
$\dtGP$ on the space $\Mtwo$. At the end of this section we prove that
$\dtGP$ is indeed a metric and that $(\Mtwo, \dtGP)$ is a Polish
metric space. The topology induced by the metric $\dtGP$ will turn out
to be the same as the two-level Gromov-weak topology (\cf
Section~\ref{sec:EquivalenceOfBothTopologies}).

\begin{definition}[Two-level Gromov-Prokhorov metric]\label{def:2GPmetric}
  Let $\calX = (X, r, \nu)$ and $\calY = (Y, d, \lambda)$ be two m2m
  spaces. We define the \emph{two-level Gromov-Prokhorov distance}
  $\dtGP(\calX, \calY)$ between $\calX$ and $\calY$ by
  \begin{equation}\label{eq:def:2GPmetric}
    \dtGP(\calX, \calY) \asn \inf_{Z, \iota_X, \iota_Y}
    \dP^{\calM_f(Z)}(\iota_{X**}\nu, \iota_{Y**}\lambda),
  \end{equation}
  where the infimum ranges over all isometric embeddings $\iota_X
  \colon X \to Z$, $\iota_Y \colon Y \to Z$ into a common Polish
  metric space $Z$ and where $\dP^{\calM_f(Z)}$ denotes the Prokhorov
  metric for measures on the space $\calM_f(Z)$. The topology
  induced by this metric is called the \emph{two-level Gromov-Prokhorov
    topology} and denoted by $\tautGP$.
\end{definition}
Recall from Corollary~\ref{cor:SupportOfMomentMeasure} that for any
m2m space $(X, r, \nu)$ the support of $\nu$ is a subset of
\begin{align*}
\set{\mu \in \calM_f(X) | \supp\mu \subset \supp\mm{\nu}}.
\end{align*}
With this fact it is easy to see that the two-level Gromov-Prokhorov
distance between equivalent m2m spaces is zero. Thus, the value of
$\dtGP(\calX, \calY)$ does not depend on the chosen representatives of
the equivalence classes and is well-defined. In Proposition
\ref{prp:dtGPIsPolishMetric} we prove that $\dtGP$ is in fact a metric
and even complete.

Finding a ``good'' space $Z$ and embeddings $\iota_X, \iota_Y$ in
\eqref{eq:def:2GPmetric} can be a challenging task. If $X$ and $Y$
have a similar structure, there might be a natural way to overlap $X$
and $Y$ in such a way that the Prokhorov distance in
\eqref{eq:def:2GPmetric} is small. However, if $X$ and $Y$ are very
different, it might be better to choose $Z$ as the disjoint union
$X \sqcup Y$. Then, the problem of finding an appropriate space $Z$ and
embeddings $\iota_X, \iota_Y$ reduces to finding a metric $r'$ on
$X \sqcup Y$ such that it extends the old metrics $r$ and $d$ and
connects the spaces $X$ and $Y$ in an optimal way. In the next lemma
we prove that this procedure gives the same results as in the original
definition in~\eqref{eq:def:2GPmetric}.
\begin{lemma}[Alternative definition of $\dtGP$]\label{lm:2GPmetricEquiv}
  For all m2m spaces $\calX = (X, r, \nu)$ and
  $\calY = (Y, d, \lambda)$ we have
   \begin{displaymath}
    \dtGP(\calX, \calY) = \inf_{r'}
    \dP^{\calM_f(X \sqcup Y, r')}(\nu, \lambda),
  \end{displaymath}
  where the infimum ranges over all metrics $r'$ on the disjoint union
  $X \sqcup Y$ which extend the metrics $r$ and $d$.
\end{lemma}
{
  Observe that $(X \sqcup Y, r')$ is again a Polish metric space since
  separability and completeness are inherited from the spaces $X$ and
  $Y$: If $X_c$ and $Y_c$ are countable dense subsets of $X$ and $Y$,
  respectively, then $X_c \sqcup Y_c$ is countable and dense in
  $X \sqcup Y$. To show completeness, let $(z_n)_n$ be a Cauchy
  sequence in $(X \sqcup Y, r')$. Then there must be a subsequence
  $(z_n')_n$ with $\set{z_n' | n \in \N} \subset X$ or
  $\set{z_n' | n \in \N} \subset Y$. Since $r'$ extends the complete
  metrics $r$ and $d$, this subsequence must be convergent. Hence the
  original Cauchy sequence $(z_n)_n$ is also convergent.
}

In the preceding lemma we regard $\nu$ and $\lambda$ as elements of
the same space $\calM_f(\calM_f(X \sqcup Y))$ without using the
canonical embeddings $\iota_X \from X \to X \sqcup Y$ and
$\iota_Y \from Y \to X \sqcup Y$. We will continue to use this
abbreviated notation in the remainder of this section whenever it
seems appropriate (\ie when we are embedding spaces into their
disjoint union). In the same manner we will regard points $x \in X$
and $y \in Y$ as elements of $X \sqcup Y$ and write $r'(x,y)$ instead
of $r'(\iota_{X}(x), \iota_{Y}(y))$ (where $r'$ is a metric on
$X \sqcup Y$).

\begin{proof}
  By Definition~\ref{def:2GPmetric} we always have
  \begin{displaymath}
    \dtGP(\calX, \calY) \leq \inf_{r'}
    \dP^{\calM_f(X \sqcup Y, r')}(\nu, \lambda).
  \end{displaymath}
  On the other hand, let $\dtGP(\calX, \calY) < \epsilon$. Then there
  is a Polish metric space $(Z, r_Z)$ and isometric embeddings
  $\iota_X \from X \to Z$, $\iota_Y \from Y \to Z$ such that
  \begin{displaymath}
    \dP^{\calM_f(Z, r_Z)}(\iota_{X**}\nu, \iota_{Y**}\lambda) < \epsilon.
  \end{displaymath}
  For $\delta>0$ define $r'_\delta$ as the metric on $X \sqcup Y$
  which extends $r$ and $d$ and satisfies
  \begin{displaymath}
     r'_\delta(x, y) = \delta + r_Z(\iota_X(x), \iota_Y(y))
  \end{displaymath}
  for every $x \in X$ and $y \in Y$. Then $(X \sqcup Y, r'_\delta)$ is complete and
  separable and we have
  \begin{displaymath}
    \dP^{\calM_f(X \sqcup Y, r'_\delta)}(\nu, \lambda)
    < \epsilon + \delta.
  \end{displaymath}
  Ranging over all possible $\epsilon$ and $\delta$ yields the claim.
\end{proof}

In Lemma~\ref{lm:d2GPconvergence} we will see that a sequence
$(X_n, r_n, \nu_n)$ of m2m spaces converges to a limit m2m space
$(X, r, \nu)$ with respect to $\dtGP$ if and only if we can embed all
the spaces isometrically in a common Polish metric space $Z$ such
that the two-level push-forward of $\nu_n$ converges weakly in
$\calM_f(\calM_f(Z))$ to the two-level push-forward of $\nu$. A similar result
holds for Cauchy sequences of m2m spaces as can be seen in the
following Lemma.
\begin{lemma}[Embedding of sequences of m2m spaces]\label{lm:2GPcauchysequence}
  Let $(\epsilon_n)_n$ be a sequence of positive real numbers and let
  $(\calX_n)_n$ be a sequence of m2m spaces with
  $\calX_n = (X_n, r_n, \nu_n)$ and
  \begin{displaymath}
    \dtGP(\calX_n, \calX_{n+1}) < \epsilon_n
  \end{displaymath}
  for every $n \in \N$. Then there is a Polish metric space $(Z, r_Z)$
  and isometric embeddings $\iota_1, \iota_2, \dotsc$ of $X_1, X_2,
  \dotsc$, respectively, into $Z$ such that
  \begin{displaymath}
    \dP^{\calM_f(Z, r_Z)}(\iota_{n**}\nu_n, \iota_{n+1**}\nu_{n+1}) <
    \epsilon_n
  \end{displaymath}
  for all $n \in \N$.
\end{lemma}
\begin{proof}
  We define $Z_n \asn \bigsqcup_{k=1}^n X_k$ for every $n \in \N$ and
  $Z_\infty \asn \bigsqcup_{k=1}^\infty X_k$. We will inductively
  define metrics $d_n$ on $Z_n$ using Lemma~\ref{lm:2GPmetricEquiv}.
  By this lemma there exists a metric $d_2$ on $Z_2 = X_1 \sqcup X_2$
  such that $(Z_2, d_2)$ is complete and separable and
 \begin{displaymath}
    \dP^{\calM_f(Z_2, d_2)}(\nu_1, \nu_2) < \epsilon_1.
  \end{displaymath}
  This metric extends $r_1$ and $r_2$ and therefore we have $(Z_2,
  d_2, \nu_2) \cong (X_2, r_2, \nu_2)$ and
  \begin{displaymath}
    \dtGP((Z_2, d_2, \nu_2), \calX_3) = \dtGP(\calX_2, \calX_3) < \epsilon_2.
  \end{displaymath}
  By using Lemma~\ref{lm:2GPmetricEquiv} again we find a metric $d_3$
  on $Z_3 = Z_2 \sqcup X_3$ such that $(Z_3, d_3)$ is complete and
  separable and
  \begin{align*}
    \dP^{\calM_f(Z_3, d_3)}(\nu_2, \nu_3) < \epsilon_2.
  \end{align*}
  This procedure can be continued ad infinitum and in this way we get
  a metric $d_\infty$ on $Z_\infty$ as a ``limit''. The metric space
  $(Z_\infty, d_\infty)$ is separable but not necessarily complete.
  For this reason we define $(Z, r_Z)$ as the completion of
  $(Z_\infty, d_\infty)$. It is easily seen that this completion has
  the desired properties with $\iota_n$ as the canonical inclusion of
  $X_n$ into $Z_\infty \subset Z$ for every $n \in \N$.
\end{proof}

\begin{lemma}[Embedding of $\dtGP$-convergent
  sequences]\label{lm:d2GPconvergence}
  Let $(\calX_n)_n$ be a sequence of m2m spaces with
  $\calX_n = (X_n, r_n, \nu_n)$ which converges to an m2m space
  $\calX = (X, r, \nu)$. Then, there is a Polish metric space
  $(Z, r_Z)$ and isometric embeddings
  $\iota, \iota_1, \iota_2, \dotsc$ of $X, X_1, X_2, \dotsc$,
  respectively, into Z such that
  \begin{displaymath}
    \dP^{\calM_f(Z, r_Z)}(\iota_{n**}\nu_n, \iota_{**}\nu) \to 0.
  \end{displaymath}
\end{lemma}
\begin{proof}
  The proof is similar to the proof of
  Lemma~\ref{lm:2GPcauchysequence}, but this time we use $Z_n \asn
  \bigsqcup_{k=1}^n X_k \sqcup X$ and define the metrics $d_n$
  inductively by always ``routing'' through the space $X$.
\end{proof}

With the preceding embedding lemma it is easy to show that
$\tautGP$-convergence implies $\tautGw$-convergence.
\begin{lemma}[$\tautGP$ is finer than
  $\tautGw$]\label{lm:d2GPconvergenceImpliesWeakConvergence}
  Every test function $\Phi \in \tfset$ is continuous with respect
  to the two-level Gromov-Prokhorov topology. Therefore, two-level
  Gromov-Prokhorov convergence implies two-level Gromov-weak
  convergence and $\tautGP$ is finer than $\tautGw$.
\end{lemma}
\begin{proof}
  Let $\calX = (X, r, \nu), \calX_1 = (X_1, r_1, \nu_1), \dotsc$ be
  m2m spaces with $\dtGP(\calX_n, \calX) \to 0$. By
  Lemma~\ref{lm:d2GPconvergence} we may assume without loss of
  generality that all metric spaces coincide, \ie
  $(Z, r_Z) = (X, r) = (X_1, r_1) = \dotsc$, and that we have
  $\dP^{\calM_f(Z, r_Z)}(\nu_n, \nu) \to 0$. For every
  $\Phi \in \tfset$ the function
  $\lambda \mapsto \Phi((Z, r_Z, \lambda))$ from $\calM_f(\calM_f(X))$
  to $\R$ is continuous with respect to weak convergence (\cf
  Remark~\ref{rm:WeakConvImpliestGwConv}) and thus we get
  $\Phi(\calX_n) \to \Phi(\calX)$.
\end{proof}

Finally, we are able to show that the two-level Gromov-Prokhorov
distance satisfies the axioms of a metric.
\begin{proposition}\label{prp:dtGPIsPolishMetric}
  $\dtGP$ is a metric and $(\Mtwo, \dtGP)$ is a Polish metric space.
\end{proposition}
\begin{proof}
  First we prove that $\dtGP$ is indeed a metric on $\Mtwo$. Obviously
  $\dtGP$ is symmetric and non-negative. To prove the triangle
  inequality, let $\calX_i = (X_i, r_i, \nu_i) \in \Mtwo$ for $i \in
  \set{1, 2, 3}$. Let $\epsilon_{1}, \epsilon_{2} > 0$ such
  that
  \begin{align*}
    \dtGP(\calX_1, \calX_2) < \epsilon_{1}  \quad\text{and}\quad 
    \dtGP(\calX_2, \calX_3) < \epsilon_{2}.
  \end{align*}
  Then we can find a metric $d_3$ on $Z_3 \asn X_1 \sqcup X_2 \sqcup
  X_3$ that extends $r_1$, $r_2$ and $r_3$ and satisfies
  \begin{align*}
    \dP^{\calM_f(Z_3, d_3)}(\nu_1, \nu_2) < \epsilon_{1}
    \quad\text{and}\quad
    \dP^{\calM_f(Z_3, d_3)}(\nu_2, \nu_3) < \epsilon_{2}.
  \end{align*}
  Such a metric has already been constructed in the proof of
  Lemma~\ref{lm:2GPcauchysequence}. By the triangle inequality of the
  Prokhorov metric we get
  \begin{align*}
    \dtGP(\calX_1, \calX_3) \leq \dP^{\calM_f(Z_3, d_3)}(\nu_1, \nu_3)
    < \epsilon_{1} + \epsilon_{2}.
  \end{align*}
  The desired triangle inequality for $\dtGP$ follows simply by taking
  the infimum over all possible $\epsilon_{1}$ and $\epsilon_{2}$.

  Now assume that $\dtGP(\calX, \calY) = 0$ for two m2m spaces
  $\calX = (X, r, \nu)$, $\calY = (Y, d, \lambda)$. To prove that both
  spaces are equivalent, it suffices to show that
  $\Phi(\calX) = \Phi(\calY)$ for all test functions $\Phi \in \tfset$
  (\cf Theorem \ref{th:ReconstructionTheorem}). But this follows
  immediately from
  Lemma~\ref{lm:d2GPconvergenceImpliesWeakConvergence} together with
  Lemma~\ref{lm:d2GPconvergence} (by choosing $\calX_n = \calY$ for
  every $n$).

  This shows that $\dtGP$ is indeed a metric. To prove that $\dtGP$ is
  a complete metric, let $\calX_n = (X_n, r_n, \nu_n)$ be a Cauchy
  sequence with respect to $\dtGP$. By
  Lemma~\ref{lm:2GPcauchysequence} we can embed the metric spaces
  $(X_n, r_n)_n$ into a common Polish metric space $(Z, r)$ using
  isometries $(\iota_n)_n$. The two-level push-forward measures
  $(\iota_{n**}\nu_n)_n$ form a Cauchy sequence in
  $\calM_f(\calM_f(Z))$ and thus converge weakly to some
  $\nu \in \calM_f(\calM_f(Z))$. It follows that $(\calX_n)_n$
  converges to the m2m space $(Z, r, \nu)$ with respect to the two-level
  Gromov-Prokhorov metric.

  To show that $(\Mtwo, \dtGP)$ is separable, we define $\S$ as the
  set of all m2m spaces $(X, r, \nu) \in \Mtwo$ such that
  $\abs{X} < \infty$, the metric $r$ takes only rational values and
  \begin{align}\label{eq:prp:dtGPIsPolishMetric1}
    \nu = \sum_{i=1}^{M} a_i \delta_{\left(\sum_{j=1}^{N_i} b_{ij}
      \delta_{x_{ij}}\right)} &\text{ with } x_{ij} \in X,\ a_i, b_{ij} \in
    \Qplus,\ M, N_i \in \N.
  \end{align}
  That is, $(X, r, \nu)$ is a finite m2m space with only rational
  distances and $\nu$ is a finite atomic measure on finite atomic
  measures with only rational values. The set $\S$ is obviously
  countable. To prove density, let $(X, r, \lambda)$ be an arbitrary
  m2m space and let $\epsilon>0$. Because the set of measures of the
  form~\eqref{eq:prp:dtGPIsPolishMetric1} is dense in
  $\calM_f(\calM_f(X))$, there is a $\nu \in \calM_f(\calM_f(X))$ of
  this form with $\dP(\lambda, \nu) < \frac{\epsilon}{2}$. Note that
  $S \asn \supp\mm{\nu}$ is finite, thus $(S, r, \nu)$ is a finite m2m
  space that is $\frac{\epsilon}{2}$-close to $(X, r, \lambda)$. The
  last step is to approximate $r$ by a rational version $r'$ such that
  $\abs{r(x, y) - r'(x, y)} < \frac{\epsilon}{2}$ for all
  $x, y \in S$. Then $(S, r', \nu)$ is in $\S$ and we have
  $\dtGP( (X, r, \lambda), (S, r', \nu)) < \epsilon$.
\end{proof}

\section{Distance distribution and modulus of mass distribution}
\label{sec:DDandMMD}

In this section we define the distance distribution and the modulus
of mass distribution, which have been introduced in \cite{GPW09}. They
are indicators for the complexity of a metric measure space
$(X, r, \mu)$ and will be used in the characterization of compactness
in Section~\ref{sec:compactness}. 
\begin{definition}[Distance distribution and modulus of mass distribution]
  Let $\mu$ be a finite Borel measure on a Polish metric space
  $(X, r)$.
  \begin{enumerate}
    \item The \emph{distance distribution
      $\DD{\mu} \in \calM_f(\Rplus)$ of $\mu$} is defined by
    $\DD{\mu} \asn r_*\mu^{\otimes 2}$.
    \item For $\delta \geq 0$ the \emph{modulus of mass distribution
    $\MMD{\mu}$ of $\mu$} is the number defined by
    \begin{align*}
      \MMD{\mu} \asn \inf\set{ \epsilon>0 | \mu(\set{ x \in X |
          \mu(B(x, \epsilon)) \leq \delta}) \leq \epsilon}.
    \end{align*}
  \end{enumerate}
\end{definition}
The distance distribution reflects the effective diameter of the
support of $\mu$, whereas the modulus of mass distribution measures
the fineness of the measure $\mu$. Heuristically, $\MMD{\mu}$ is the
amount of ``thin points'' of $X$, where we think of $x \in X$ as a
thin point if $\mu(B(x, \epsilon)) \leq \delta$ for given
$\epsilon, \delta > 0$.

Note that the values of the distance distribution and the modulus of
mass distribution coincide for measures from equivalent mm spaces (\cf
\cite[Remark 2.10]{GPW09}).

The distance distribution and the modulus of mass distribution are
constructed for one-level measures $\mu \in \calM_f(X)$. At first sight
they seem to be inappropriate to measure the complexity of a two-level
measure $\nu \in \calM_f(\calM_f(X))$.
However, we can overcome this problem by ``projecting'' the two levels
of $\nu$ to one level, \ie by looking at the first moment measure
$\mm{\nu}(\wildcard) = \int \mu(\wildcard) \dif\nu(\mu)$. As mentioned
earlier, the moment measure of a two-level measure is not necessarily
finite. This is why we will approximate $\nu$ by measures from
$\calM_f(\calM_{\leq K}(X))$ with $K \nearrow \infty$. This
approximation will be introduced in the next section.

In the rest of this section we summarize some useful properties
of the modulus of mass distribution.
\begin{lemma}[Properties of the modulus of mass
  distribution]\label{lm:MMDProperties}
  Let $\mu$ be a finite Borel measure on a Polish metric space
  $(X, r)$.
  \begin{enumerate}
    \item\label{it:lm:MMDProperties_1} The function
    $\delta \mapsto \MMD{\mu}$ is non-decreasing and bounded by the
    total mass $\mass{\mu}$. Moreover, we have
    $\lim_{\delta \searrow 0}\MMD{\mu} = 0$.
    \item\label{it:lm:MMDProperties_2} For $\epsilon, \delta > 0$ we
    have $\MMD{\mu} < \epsilon$ if and only if
    $\mu(\set{x \in X | \mu(B(x, \epsilon)) \leq \delta }) <
    \epsilon$.
    \item\label{it:lm:MMDProperties_3} Let $\epsilon$ and $\delta$ be
    positive real numbers with $\MMD{\mu} < \epsilon$. Then there is a
    finite set $A \subset X$ with
    $\abs{A} \leq \max(1, \frac{\mass{\mu}}{\delta})$ such that
$
      \mu(\Complement B(A, \epsilon)) <  \epsilon.
$
  \end{enumerate}
\end{lemma}
\begin{proof}
  \begin{enumerate}
    \item The claim was proved for probability measures in \cite[Lemma
    6.5]{GPW09}. The same proof holds true for finite measures.
    \item For the ``only if direction'' see Lemma 6.4 in \cite{GPW09}.
    To prove the other direction, observe that the function
    \begin{align*}
      \epsilon \mapsto \mu(B(x, \epsilon))
    \end{align*}
    is left-continuous for every $x \in X$ and that we have the equality
    \begin{align*}
      \mu(\set{x \in X | \mu(B(x, \epsilon)) \leq \delta}) = \int
      \One_{[0, \delta]}(\mu(B(x, \epsilon))) \dif\mu(x).
  \end{align*}
  It follows from the dominated convergence theorem that the function
  \begin{align*}
    \epsilon \mapsto \mu(\set{x \in X | \mu(B(x, \epsilon)) \leq \delta})
  \end{align*}
  is also left-continuous for a fixed $\delta$. If $V_\delta(\mu)
  \geq \epsilon$, then we have for every $\epsilon' < \epsilon$ that
  $\mu(\set{x \in X | \mu(B(x, \epsilon')) \leq \delta}) > \epsilon'$
  and therefore
  \begin{align*}
    \mu(\set{x \in X | \mu(B(x, \epsilon)) \leq \delta }) =
    \lim_{\epsilon' \nearrow \epsilon} \mu(\set{x \in X | \mu(B(x,
    \epsilon')) \leq \delta} \geq \epsilon.
  \end{align*}
  \item The assertion was proved in \cite[Lemma 6.9]{GPW09}, but only
  for probability measures. This is why we give a full proof for
  finite measures: In case $\mass{\mu} \leq \epsilon$ {we have
  $\mu(\Complement B(x, \epsilon)) < \epsilon$ for every
  $x \in \supp \mu$ (or every $x \in X$ if $\mu = \nullmeasure$)} and
  we are done. Otherwise we define
  $D \asn \set{ x \in X | \mu(B(x, \epsilon)) > \delta}$. Because
  $\MMD{\mu}$ is less than $\epsilon$, we have
  $\mu(\Complement D) < \epsilon < \mass{\mu}$ and $D$ is not empty.
  By \cite[page 278]{BuragoBurago} there exists an
  $\epsilon$-separated discrete subset $A$ of $D$ that is maximal.
  That is, we have $r(x_1, x_2) \geq \epsilon$ for $x_1, x_2 \in A$
  with $x_1 \not= x_2$ and adding a further point of $D$ would destroy
  this property. It follows from the maximality that
  $D \subset B(A, \epsilon)$ and therefore
  $\mu(\Complement B(A, \epsilon)) \leq \mu(\Complement D) <
  \epsilon$. Moreover, we see that
  \begin{displaymath}
    \mass{\mu} \geq \mu\left(B(A, \epsilon)\right)
    = \sum_{x \in A}\mu\left(B(x, \epsilon)\right) \geq \abs{A} \delta,
  \end{displaymath}
  which yields the claim.
  \end{enumerate}
\end{proof}

\begin{lemma}[$\MMD{}$ is upper
  semi-continuous]\label{lm:MMDUpperSemiCts}
  Let $\M$ denote the set of metric measure spaces equipped with the
  Gromov-weak topology and let $\delta>0$ be fixed. The function
  \begin{align*}
    \M &\to \Rplus \\
    (X, r, \mu) &\mapsto \MMD{\mu} 
  \end{align*}
  is upper semi-continuous. That is, if a net (or a sequence)
  $((X_\alpha, r_\alpha, \mu_\alpha))_\alpha$ of metric measure spaces converges Gromov
  weakly to $(X, r, \mu)$, then
  \begin{displaymath}
    \limsup_{\alpha}\MMD{\mu_\alpha} \leq \MMD{\mu}.
  \end{displaymath}
\end{lemma}
\begin{proof}
  The proof for sequences of metric probability measure spaces can be
  found in \cite[Proposition 6.6]{GPW09}. It remains valid even if we
  replace metric probability measure spaces by metric measure
  spaces with \emph{finite} measures.
\end{proof}

The following technical lemma is a preparation for the example in
Section~\ref{sec:NestedKingman}.
\begin{lemma}\label{lm:AMMDLowerSemiCts}
  Let $\M$ denote the set of metric measure spaces equipped with the
  Gromov-weak topology and let $\epsilon, \delta>0$ be fixed. The
  function
  \begin{align*}
    \M &\to \Rplus \\
    (X, r, \mu) &\mapsto \mu(\set{x \in X | \mu(\clBall(x, \epsilon))
    < \delta})
  \end{align*}
  is lower semi-continuous. That is, if a net (or a sequence)
  $((X_\alpha, r_\alpha, \mu_\alpha))_\alpha$ of metric measure spaces converges Gromov
  weakly to $(X, r, \mu)$, then
\begin{align*}
  \liminf_{\alpha}\mu_\alpha(\set{x \in X |
  \mu_\alpha(\clBall(x, \epsilon)) < \delta}) \geq \mu({x \in X |
  \mu(\clBall(x, \epsilon)) < \delta}).
\end{align*}
\end{lemma}
\begin{proof}
  The proof of \cite[Proposition 6.6 (iv)]{GPW09} actually shows that
  \begin{align*}
    \M &\to \Rplus \\
    (X, r, \mu) &\mapsto \mu(\set{x \in X | \mu(B(x, \epsilon))
    \leq \delta})
  \end{align*}
  is upper semi-continuous by using the Portmanteau Theorem for closed
  sets. The proof of our claim is similar, but uses the Portmanteau
  Theorem for open sets instead.
\end{proof}

\section{Approximation of m2m spaces}
\label{sec:Approximation}

As mentioned before, we need to approximate two-level measures
$\nu \in \calM_f(\calM_f(X))$ by measures $\nu_K$ from
$\calM_f(\calM_{\leq K}(X))$ to ensure that the first moment measure
is finite. The simplest choice for $\nu_K$ would be the restriction of
$\nu$ to $\calM_{\leq K}(X)$, \ie
$\nu_K(\wildcard) = \nu(\wildcard \cap \calM_{\leq K}(X))$. However,
this rough cut-off may lead to discontinuities if $\nu(\calM_K(X))$
is greater than 0. Therefore it is better to ``cut off'' $\nu$ with a
continuous density $f_K$ as defined below.

Let $\set{g_K \in \calC_b(\Rplus) | K > 0}$ be a set of functions
having the following properties:
\begin{enumerate}
  \item $0 \leq g_K \leq 1$ for every $K$,
  \item $g_K(x) = 0$ for $x \geq K$,
  \item $g_K \to 1$ for $K \to \infty$ uniformly on every bounded
  interval.
\end{enumerate}
For example we may define
\begin{align*}
  g_K(x) \asn
  \begin{cases}
    1 & 0 \leq x \leq \frac{K}{2} \\
    2-\frac{2x}{K} & \frac{K}{2} < x \leq K \\
    0 & K < x.
  \end{cases}
\end{align*}

Moreover, let $f_K \asn g_K \circ \mass{}$. For a two-level measure
$\nu \in \calM_f(\calM_f(X))$ we denote by $f_K \cdot \nu$ the measure
that has density $f_K$ with respect to $\nu$. That is, it is the
unique measure which satisfies
\begin{align*}
  f_K \cdot \nu(A) = \int_A f_K(\mu) \dif\nu(\mu) = \int_A
  g_K(\mass{\mu}) \dif\nu(\mu)
\end{align*}
for every measurable $A \subset \calM_f(X)$. The measures
$\set{ f_K\cdot \nu | K>0}$ will serve as approximations of
$\nu$. Clearly we have $f_K \cdot \nu \wto \nu$ for $K \to
\infty$. Moreover, if $\nu_n \wto \nu$, we have $f_K \cdot \nu_n \wto
f_K \cdot \nu$ for every $K>0$. Observe that $f_K\cdot \nu \in \calM_f(\calM_{\leq
  K}(X))$, thus the first moment measure is finite.

\begin{lemma}[Properties of the approximation
  $f_K \cdot \nu$]\label{lm:fKisContinuous}
  The following statements hold true in the two-level Gromov-Prokhorov
  and in the two-level Gromov-weak topology:
  \begin{enumerate}
    \item\label{it:lm:fKisContinuous1} The function $(X, r, \nu)
    \mapsto (X, r, f_K\cdot\nu)$ is continuous for every $K>0$.
    \item\label{it:lm:fKisContinuous1a} For every $K>0$ the function
    $(X, r, \nu) \mapsto (X, r, \mm{f_K\cdot\nu})$ is continuous from
    $\Mtwo$ to $\M$, where $\M$ denotes the set of
    metric measure spaces equipped with the Gromov-weak topology.
    \item\label{it:lm:fKisContinuous1b} The function $(X, r, \nu)
    \mapsto \DD{\mm{f_K \cdot \nu}}$ is continuous for every $K>0$.
    \item\label{it:lm:fKisContinuous2} $(X, r, f_K \cdot \nu) \to (X,
    r, \nu)$ for $K \to \infty$ and for every m2m space $(X, r, \nu)
    \in \Mtwo$.
  \end{enumerate}
\end{lemma}
\begin{proof}
  \ref{it:lm:fKisContinuous1} with $\tautGw$: Fix $K>0$ and let the
  net $((X_\alpha, r_\alpha, \nu_\alpha))_\alpha$ converge to
  $(X, r, \nu)$ in the two-level Gromov-weak topology. We use
  Lemma~\ref{lm:testfunctions_alternative} to show that
  $(X_\alpha, r_\alpha, f_K \cdot \nu_\alpha)$ converges to
  $(X, r, f_K \cdot \nu)$. Since
  $\mass{}_*\nu_\alpha \wto \mass{}_*\nu$ (\cf
  Remark~\ref{rm:MassIs2lvlGromovWeakContinuous}) and $g_K$ is
  continuous and bounded, we have for every $h \in \calC_b(\Rplus)$
  \begin{align*}
    \int h \dif\mass{}_*(f_K \cdot \nu_\alpha) = \int h(\mass{\mu})
    g_K(\mass{\mu}) \dif\nu_\alpha(\mu) = \int h(m)g_K(m)
    \dif\mass{}_*\nu_\alpha(m)
  \end{align*}
  and this converges to
  \begin{align*}
    \int h(m) g_K(m) \dif\mass{}_*\nu(m) = \int h \dif\mass{}_*(f_K\cdot\nu).
  \end{align*}
  It follows that $\mass{}_*(f_K\cdot\nu_\alpha)$ converges weakly to
  $\mass{}_*(f_K\cdot\nu)$. Now let $\tilde{\Phi}$ be as in
  \eqref{TF4}. Then
  \begin{align*}
    \tilde{\Phi}((X_\alpha, r_\alpha, f_K \cdot \nu_\alpha)) &= \int \psi(\mass{\vect{\mu}}) \int
    \phi \circ R \dif \normalize{\vect{\mu}}^{\otimes\vect{n}}
    \dif (f_K\cdot\nu_\alpha)^{\otimes m}(\vect{\mu}) \\
    &= \int \psi(\mass{\vect{\mu}}) \prod_{i=1}^m g_K(\mass{\mu_i}) \int
    \phi \circ R \dif \normalize{\vect{\mu}}^{\otimes\vect{n}}
    \dif \nu_\alpha^{\otimes m}(\vect{\mu})
  \end{align*}
  and this converges to
  \begin{align*}
    \int \psi(\mass{\vect{\mu}}) \prod_{i=1}^m g_K(\mass{\mu_i}) \int
    \phi \circ R \dif \normalize{\vect{\mu}}^{\otimes\vect{n}}
    \dif \nu^{\otimes m}(\vect{\mu}) = \tilde{\Phi}((X, r, f_K \cdot \nu)).
  \end{align*}
  Therefore, both conditions of
  Lemma~\ref{lm:testfunctions_alternative} are satisfied and
  $(X_\alpha, r_\alpha, f_K \cdot \nu_\alpha)$ converges to
  $(X, r, f_K \cdot \nu)$.

  \ref{it:lm:fKisContinuous1} with $\tautGP$: If we have a converging
  sequence of m2m spaces $(X_n, r_n, \nu_n)\to (X, r, \nu)$, we can
  embed all the metric spaces isometrically into a common Polish
  metric space $(Z, r_Z$) such that the measures $\nu_n$ converge
  weakly to $\nu$ in $\calM_f(\calM_f(Z))$ (\cf
  Lemma~\ref{lm:d2GPconvergence}). Observe that the function
  $\nu \to f_K\cdot\nu$ is weakly continuous on $\calM_f(\calM_f(Z))$.
  Thus, $f_K\cdot\nu_n$ converges weakly to $f_K\cdot\nu$ in
  $\calM_f(\calM_f(Z))$ and weak convergence implies convergence of
  the corresponding m2m spaces in the two-level Gromov-Prokhorov
  metric.

  \ref{it:lm:fKisContinuous1a}: Recall from
  Lemma~\ref{lm:d2GPconvergenceImpliesWeakConvergence} that the
  two-level Gromov-weak topology $\tautGw$ is coarser than the
  two-level Gromov-Prokhorov topology $\tautGP$. Thus, it suffices
  to show continuity only with respect to $\tautGw$. Let $K>0$ and let
  the net $((X_\alpha, r_\alpha, \nu_\alpha))_\alpha$ converge to
  $(X, r, \nu)$ in the two-level Gromov-weak topology. By
  assertion~\ref{it:lm:fKisContinuous1}
  $(X_\alpha, r_\alpha, f_K \cdot \nu_\alpha)$ converges two-level
  Gromov-weakly to $(X, r, f_K \cdot \nu)$. We want to show that the
  mm spaces $((X_\alpha, r_\alpha, \mm{f_K \cdot \nu_\alpha}))_\alpha$
  converge Gromov-weakly to $(X, r, \mm{f_K\cdot\nu})$. By the
  Definition of the Gromov-weak topology (\cf \cite{GPW09}) we need to
  show that
  $\hat{\Phi}((X_\alpha, r_\alpha, \mm{f_K \cdot \nu_\alpha}))$
  converges to $\hat{\Phi}((X, r, \mm{f_K\cdot\nu}))$ for every test
  function $\hat{\Phi} \from \M \to \Rplus$ of the form
  \begin{align*}
    \hat{\Phi}((X, r, \mu)) = \int \phi \circ R \dif\mu^{\otimes m}
  \end{align*}
  with $m \in \N$ and $\phi \in \calC_b(\Rplus^{m \times m})$. Fix
  such a $\hat{\Phi}$ and let $\vect{n} \asn (1, \dotsc, 1) \in \N^m$.
  Then
  \begin{align*}
    \hat{\Phi}((X_\alpha, r_\alpha, \mm{f_K \cdot \nu_\alpha})) =  \int
    \phi \circ R\, \dif(\mm{f_K \cdot \nu_\alpha})^{\otimes
    m} = \int \int \phi \circ R\,
    \dif\vect{\mu}^{\otimes\vect{n}} \dif (f_K \cdot
    \nu_\alpha)^{\otimes m}(\vect{\mu}). 
  \end{align*}
  If we choose a function $\psi \in \calC_b(\Rplus^m)$ such that
  $\psi(x_1, \dotsc x_m) = \prod_{i=1}^m x_i$ on $[0, K]^m$, the
  right hand side of the last equation can be written as
  \begin{displaymath}
    \tilde{\Phi}((X_\alpha,  r_\alpha, f_K \cdot \nu_\alpha)) \asn
    \int \psi(\mass{\vect{\mu}}) \int \phi
    \circ R \dif\normalize{\vect{\mu}}^{\otimes\vect{n}}
    \dif (f_K\cdot\nu_\alpha)^{\otimes m}(\vect{\mu}).
  \end{displaymath}
  $\tilde{\Phi}$ is a test function as in \eqref{TF4}. Since
  $(X_\alpha, r_\alpha, f_K \cdot \nu_\alpha)$ converges
  two-level Gromov-weakly, we also have convergence of all test
  functions of this form (by Lemma~\ref{lm:testfunctions_alternative})
  and thus
  \begin{align*}
    \hat{\Phi}((X_\alpha, r_\alpha, \mm{f_K \cdot \nu_\alpha})) =
    \tilde{\Phi}((X_\alpha, r_\alpha, f_K \cdot \nu_\alpha)) \to
    \tilde{\Phi}((X, r, f_K 
    \cdot \nu)) = \hat{\Phi}((X, r, \mm{f_K \cdot \nu})).
  \end{align*}
 
  \ref{it:lm:fKisContinuous1b}: The claim follows immediately with
  assertion~\ref{it:lm:fKisContinuous1a} and the fact that the function
  \begin{align*}
    \M &\to \calM_1(\Rplus) \\
    (X, r, \mu) &\mapsto \DD{\mu}
  \end{align*}
  is continuous (\cf \cite[Proposition 6.6]{GPW09}).  

  \ref{it:lm:fKisContinuous2}: $f_K\cdot\nu$ converges weakly to
  $\nu$ for $K \to \infty$. Therefore, $(X, r, f_K \cdot \nu)$
  converges to $(X, r, \nu)$ in the two-level Gromov-Prokhorov metric.
  Since two-level Gromov-Prokhorov convergence implies two-level Gromov
  weak convergence
  (\cf Lemma~\ref{lm:d2GPconvergenceImpliesWeakConvergence}), assertion
  \ref{it:lm:fKisContinuous2} is true for both topologies.
\end{proof}

By combining the preceding lemma with Lemma~\ref{lm:MMDUpperSemiCts}
we get the following corollary.
\begin{corollary}\label{cor:tGwConvergenceAndMMD}
  For all $\delta, K>0$ the function
  \begin{align*}
    \Mtwo &\to \Rplus \\
    (X, r, \nu) &\mapsto \MMD{\mm{f_K \cdot \nu}}
  \end{align*}
  is upper semi-continuous in the two-level Gromov-weak topology (and
  in the two-level Gromov-Prokhorov topology). That is, if a net (or a
  sequence) $((X_\alpha, r_\alpha, \nu_\alpha))_\alpha$ of m2m spaces
  converges two-level Gromov weakly to $(X, r, \nu)$, then
  \begin{displaymath}
    \limsup_{\alpha}\MMD{\mm{f_K \cdot \nu_\alpha}} \leq \MMD{\mm{f_K
        \cdot \nu}}. 
  \end{displaymath}
  In particular this implies
  $\limsup\limits_\alpha\MMD{\mm{f_K \cdot \nu_\alpha}} \to 0$ for
  $\delta \searrow 0$.
\end{corollary}

\section{Compactness}
\label{sec:compactness}

In this section we examine compactness in $(\Mtwo, \dtGP)$. The main
result of this section is Theorem~\ref{th:CompactnessEquivalencies},
in which we give several equivalent criteria for relative compactness.
In Theorem~\ref{th:CompactNETSEquivalencies} we characterize compact
nets. It might seem odd to look at nets in a metric space. However, we
need to show in the proof of
Theorem~\ref{th:dtGPandtGwConvergenceCoincide} that every
$\tautGw$-convergent net is also $\tautGP$-convergent. Thus it is
useful for us to have a better understanding about $\tautGP$-compact
nets. 

But first we introduce a sequence of compact subsets of $\Mtwo$ whose
union is dense.
  For every $N \in \N$ we define $\A_N \subset \Mtwo$ as the set of
  all m2m spaces $(X, r, \nu)$ such that 
  \begin{itemize}
    \item $\supp\mm{\nu}$ consists of at most $N$ points, 
    \item the diameter of $\supp\mm{\nu}$ is at most $N$,
    \item $\nu \in \calM_{\leq N}(\calM_{\leq N}(X))$.
  \end{itemize}
  Observe that the union $\bigcup_{N \in \N}\A_N$ is dense in
  $(\Mtwo, \dtGP)$ since it contains the dense set $\S$ from the proof
  of Proposition~\ref{prp:dtGPIsPolishMetric}.
\begin{lemma}
  $\A_N$ is compact in the two-level Gromov-Prokhorov topology for
  every $N\in\N$.
\end{lemma}
\begin{proof}
  Let $((X_n, r_n, \nu_n))_n$ be a sequence in $\A_N$. Without loss of
  generality we assume $X_n = \supp\mm{\nu_n}$. The finite metric
  spaces $((X_n, r_n))_n$ are determined by the number of points and
  the mutual distances between the points. All of these are bounded by
  $N$. By \cite[Theorem 7.4.15]{BuragoBurago} $((X_n, r_n))_n$ is
  relatively compact in the Gromov-Hausdorff topology and there is a
  subsequence which converges to some compact metric space $(X, r)$.
  For the sake of convenience we denote this subsequence again by
  $((X_n, r_n))_n$. We have $\abs{X} \leq N$ and
  $\diam X \leq N$. By \cite[Lemma A.1]{GPW09} there is a compact
  metric space $(Z, r_Z)$ and isometric embeddings
  $\iota, \iota_1, \iota_2, \dotsc$ of $X, X_1, X_2, \dotsc$ into $Z$
  such that
  \begin{displaymath}
    d_H^Z(\iota_n(X_n), \iota(X)) \to 0.
  \end{displaymath}
  Here, $d_H^Z$ denotes the Hausdorff metric on $Z$. Because $Z$ is compact,
  both $\calM_{\leq N}(Z)$ and $\calM_{\leq N}(\calM_{\leq N}(Z))$ are
  compact too. Therefore, $((\iota_{n**}\nu_n))_n$ has a subsequence
  which converges weakly to some measure
  $\nu \in \calM_{\leq N}(\calM_{\leq N}(Z))$. Then, the same
  subsequence of $((X_n, r_n, \nu_n))_n$ converges to $(Z, r_Z, \nu)$
  and $(Z, r_Z, \nu) \cong (X, r, \nu) \in \A_N$.
\end{proof}

\begin{theorem}[Characterization of compact sets]\label{th:CompactnessEquivalencies}
  Let $\Gamma \subset \Mtwo$ be a set of m2m spaces. The following are
  equivalent:
  \begin{enumerate}
    \item\label{it:th:CompactnessEquivalencies1}
    $\Gamma$ is relatively compact in the two-level Gromov-Prokhorov
    topology.
    \item\label{it:th:CompactnessEquivalencies2} For every $\epsilon >
    0$ there exists an $N \in \N$ such that $\dtGP(\calX, \A_N) <
    \epsilon$ for every $\calX \in \Gamma$.
    \item\label{it:th:CompactnessEquivalencies4}
    $\set{\mass{}_*\nu | (X, r, \nu) \in \Gamma}$ is relatively
    compact in $\calM_f(\Rplus)$ and for every $K > 0$ we have
    \begin{itemize}
      \item $\sup_{(X, r, \nu) \in \Gamma} \MMD{\mm{f_K \cdot \nu}} \to 0$
       for $\delta \searrow 0$,
      \item $\set{ \DD{\mm{f_K \cdot \nu}} | (X, r, \nu) \in \Gamma}$ is
      relatively compact in $\calM_f(\Rplus)$.
    \end{itemize}
    \item\label{it:th:CompactnessEquivalencies3}
    $\set{\mass{}_*\nu | (X, r, \nu) \in \Gamma}$ is relatively
    compact in $\calM_f(\Rplus)$ and for every $\epsilon>0$ there is
    an $N_\epsilon \in \N$ such that for every
    $\calX = (X, r, \nu) \in \Gamma$ there exists a measurable subset
    $X_{\calX, \epsilon} \subset X$ with
    \begin{itemize}
      \item
      $\nu\left( \Complement \set{ \mu \in \calM_f(X) |
          \mu(\Complement X_{\calX, \epsilon}) \leq \epsilon }\right)
      < \epsilon$,
      \item $X_{\calX, \epsilon}$ can be covered by at most
      $N_\epsilon$ balls of radius $\epsilon$,
      \item the diameter of $X_{\calX, \epsilon}$ is at most
      $N_\epsilon$.
    \end{itemize}
    \item\label{it:th:CompactnessEquivalencies5}
    $\set{\mass{}_*\nu | (X, r, \nu) \in \Gamma}$ is relatively
    compact in $\calM_f(\Rplus)$ and for every $\epsilon>0$ and
    $\calX = (X, r, \nu) \in \Gamma$ there is a compact subset
    $C_{\calX, \epsilon} \subset X$ such that
    \begin{itemize}
      \item
      $\nu\left(\Complement \set{\mu \in \calM_f(X) | \mu(\Complement
        C_{\calX, \epsilon}) \leq \epsilon}\right) < \epsilon$,
      \item
      $\calC_\epsilon \asn \set{C_{\calX, \epsilon} | \calX \in
        \Gamma}$ is relatively compact in the Gromov-Hausdorff
      topology.
    \end{itemize}
  \end{enumerate}
\end{theorem}
Note that in assertion \ref{it:th:CompactnessEquivalencies4} it
suffices to have the property only for a diverging sequence
$(K_n)_n \nearrow \infty$.

\begin{proof}[of Theorem~\ref{th:CompactnessEquivalencies}]

  The equivalence of assertions \ref{it:th:CompactnessEquivalencies1}
  to \ref{it:th:CompactnessEquivalencies3} follows easily from
  Theorem~\ref{th:CompactNETSEquivalencies} and the fact that a set is
  relatively compact if and only if every net in it is compact (\cf
  Lemma~\ref{lm:CompactNets}). In the remainder of this proof we show
  that assertions \ref{it:th:CompactnessEquivalencies3} and
  \ref{it:th:CompactnessEquivalencies5} are equivalent.

  Note that a set $\calC$ of compact metric spaces is relatively
  compact in the Gromov-Hausdorff-topology if and only if the diameter
  of the elements of $\calC$ is uniformly bounded and for every
  $\epsilon>0$ there is an $N \in \N$ such that every element of
  $\calC$ can be covered by at most $N$ balls of radius $\epsilon$.
  This fact can be found in \cite[Section 7.4.2]{BuragoBurago}.
  Therefore, assertion \ref{it:th:CompactnessEquivalencies5} readily
  implies assertion \ref{it:th:CompactnessEquivalencies3}. 

  To prove the other direction, let $\epsilon>0$. For $n \in \N$ let
  $\epsilon_n \asn \frac{\epsilon}{2} \cdot
  \left(\frac{1}{2}\right)^n$ and let $N_{\epsilon_n}$ and
  $X_{\calX, \epsilon_n}$ be as in assertion
  \ref{it:th:CompactnessEquivalencies3}. Without loss of generality we
  may assume that every $X_{\calX, \epsilon_n}$ is closed. For every
  $\calX = (X, r, \nu) \in \Gamma$ and every $n \in \N$ there is a
  compact set $\calK_{\calX, \epsilon_n} \subset \calM_f(X)$ with
  $\nu(\Complement \calK_{\calX, \epsilon_n}) < \epsilon_n$. Because
  $\calK_{\calX, \epsilon_n}$ is tight, we have
  \begin{align*}
    \calK_{\calX, \epsilon_n} \subset \set{\mu \in \calM_f(X) |
    \mu(\Complement C_{\calX, \epsilon_n}) < \epsilon_n}
  \end{align*}
  for some compact set $C_{\calX, \epsilon_n} \subset X$ and thus
  \begin{align*}
    \nu\left( \Complement \set{\mu \in \calM_f(X) | \mu(\Complement
    C_{\calX, \epsilon_n}) \leq \epsilon_n}\right) \leq
    \nu(\Complement \calK_{\calX, \epsilon_n}) < \epsilon_n. 
  \end{align*}
  Define
  $C_{\calX, \epsilon} \asn \bigcap_{n \in \N} (X_{\calX, \epsilon_n}
  \cap C_{\calX, \epsilon_n})$. Then $C_{\calX, \epsilon}$ is compact
  as a closed subset of a compact set and we have
  \begin{align*}
    \nu&\left(\Complement \set{\mu \in \calM_f(X) | \mu(\Complement
         C_{\calX, \epsilon}) \leq \epsilon}\right) \\
       &\leq \nu\left( \Complement {\textstyle\bigcap\limits_{n \in \N}} \set{\mu \in \calM_f(X) |
         \mu(\Complement X_{\calX, \epsilon_n}) \leq \epsilon_n,\
         \mu(\Complement C_{\calX, \epsilon_n}) \leq \epsilon_n }\right) \\ 
       &\leq \sum_{n \in \N} \Big(\nu\left(\Complement \set{ \mu \in
         \calM_f(X) | \mu(\Complement X_{\calX, \epsilon_n}) \leq
         \epsilon_n}\right) \\
       &\phantom{\leq}+ \nu\left( \Complement \set{\mu \in \calM_f(X) | \mu(\Complement
         C_{\calX, \epsilon_n}) \leq \epsilon_n}\right) \Big) \\
    &< 2 \sum_{n
         \in \N} \epsilon_n = \epsilon. 
  \end{align*}
  The set
  $\calC_\epsilon = \set{C_{\calX, \epsilon} | \calX \in \Gamma}$ is
  relatively compact in the Gromov-Hausdorff-topology because the
  diameter of each $C_{\calX, \epsilon}$ is bounded by
  $N_{\epsilon_1}$ and for every $\delta>0$ there is a natural number
  $N$ such that each $C_{\calX, \epsilon}$ can be covered by no more
  than $N$ balls of radius $\delta$ (\cf the remark at the beginning
  of this proof). Thus, assertion
  \ref{it:th:CompactnessEquivalencies5} is fulfilled and the proof
  is complete.
\end{proof}

\begin{theorem}[Characterization of compact
  nets]\label{th:CompactNETSEquivalencies}
  Let $(\calA, \preceq)$ be a directed set and let
  $(\calX_\alpha)_{\alpha \in \calA}$ be a net in $\Mtwo$ with
  $\calX_\alpha = (X_\alpha, r_\alpha, \nu_\alpha)$. The following are
  equivalent:
\begin{enumerate}
  \item\label{it:th:CompactNETSEquivalencies1} $(\calX_\alpha)_\alpha$
  is a compact net with respect to the two-level Gromov-Prokhorov
  topology.
  \item\label{it:th:CompactNETSEquivalencies2} For every $\epsilon>0$
  there is an $N \in \N$ such that $\dtGP(\calX_\alpha, \A_N) <
  \epsilon$ eventually.
  \item\label{it:th:CompactNETSEquivalencies3}
  $(\mass{}_*\nu_\alpha)_\alpha$ is a compact net in $\calM_f(\Rplus)$
  and for every $K > 0$ we have:
  \begin{itemize}
    \item For every $\epsilon>0$ there is a $\delta>0$ such that
    $\MMD{\mm{f_K \cdot \nu_\alpha}} < \epsilon$ eventually,
    \item $(\DD{\mm{f_K \cdot \nu_\alpha}})_\alpha$ is a compact net in
    $\calM_f(\Rplus)$.
  \end{itemize}
  \item\label{it:th:CompactNETSEquivalencies4}
  $(\mass{}_*\nu_\alpha)_\alpha$ is a compact net and for every
  $\epsilon > 0$ there is an $N_\epsilon \in \N$ such that for every
  $\alpha$ there is a set $X_{\alpha, \epsilon} \subset X_\alpha$ such
  that we eventually have
  \begin{itemize}
    \item
    $\nu_\alpha\left(\Complement\set{\mu \in \calM_f(X_\alpha) |
        \mu(\Complement X_{\alpha, \epsilon}) \leq \epsilon}\right)<
    \epsilon$,
    \item $X_{\alpha, \epsilon}$ can be covered by at most
    $N_\epsilon$ balls of radius $\epsilon$,
    \item the diameter of $X_{\alpha, \epsilon}$ is at most
    $N_\epsilon$.
  \end{itemize}
\end{enumerate}
\end{theorem}
\begin{proof}
  \ref{it:th:CompactNETSEquivalencies1} $\Rightarrow$
  \ref{it:th:CompactNETSEquivalencies2}: We prove this assertion by
  contradiction and assume \ref{it:th:CompactNETSEquivalencies2}
  does not hold. That is, there is an $\epsilon>0$ such that for every
  $N \in \N$ the distance $\dtGP(\calX_\alpha, \A_N)$ is frequently
  greater than or equal to $\epsilon$. Therefore, there is a subnet
  $(\calX_{T(\beta)})_\beta$ of $(\calX_\alpha)_\alpha$ (with a
  directed set $\calB$ and a function $T$ from $\calB$ to $\calA$) such
  that $\dtGP(\calX_{T(\beta)}, \A_N)$ is eventually at most
  $\epsilon$ for every positive integer $N$.

  Since $(\calX_{T(\beta)})_\beta$ is compact, it has a convergent
  subnet. Let $\calX$ be the limit of this subnet. Observe that the
  set $\bigcup_{N \in \N} \A_N$ is dense in $\Mtwo$ since it contains
  the dense set $\S$ from the proof of
  Proposition~\ref{prp:dtGPIsPolishMetric}. Thus, there is a natural
  number $N_0$ with $\dtGP(\calX, \A_{N_0})<\epsilon$. Consequently,
  the subnet converging to $\calX$ is eventually $\epsilon$-close to
  $\A_{N_0}$. But this contradicts the construction of the net
  $(\calX_{T(\beta)})_\beta$.

  \ref{it:th:CompactNETSEquivalencies2} $\Rightarrow$
  \ref{it:th:CompactNETSEquivalencies1}: Observe that assertion
  \ref{it:th:CompactNETSEquivalencies2} also holds for any subnet of
  $(\calX_\alpha)_\alpha$. Thus, it is enough to show that any net
  with \ref{it:th:CompactNETSEquivalencies2} has a convergent subnet.
  Inductively we can construct subnets
  $(\calX_{T_n(\beta)})_{\beta \in \calB_n}$ for each $n \in \N$ such
  that $(\calX_{T_{n}(\beta)})_{\beta \in \calB_{n}}$ is a subnet of
  $(\calX_{T_{n-1}(\beta)})_{\beta \in \calB_{n-1}}$ and eventually
  contained in a ball of radius $\nicefrac{1}{n}$. By a diagonal
  argument we can construct a subnet that is a Cauchy net. That is,
  for every $\epsilon>0$ the subnet is eventually contained in a ball
  of radius $\epsilon$. Since $(\Mtwo, \dtGP)$ is complete, this
  subnet net is convergent.

  \ref{it:th:CompactNETSEquivalencies1} $\Rightarrow$
  \ref{it:th:CompactNETSEquivalencies3}: Note that both functions
  $(X, r, \nu) \mapsto \mass{}_*\nu$ and
  $(X, r, \nu) \mapsto \DD{\mm{f_K\cdot\nu}}$ are continuous (see
  Remark~\ref{rm:MassIs2lvlGromovWeakContinuous} and
  Lemma~\ref{lm:fKisContinuous}) and that the continuous image of a
  compact net is again compact. Thus, $(\mass{}_*\nu_\alpha)_\alpha$
  and $(\DD{\mm{f_K \cdot \nu_\alpha}})_\alpha$ are both compact.

  To prove the last property, assume that it does not hold, \ie there
  are $K, \epsilon > 0$ such that
  $\MMD{\mm{f_K \cdot \nu_\alpha}} \geq \epsilon$ frequently for every
  $\delta>0$. Inductively we can construct subnets
  $(\calX_{T_n(\beta)})_{\beta \in \calB_n}$ for each $n \in \N$ such
  that $(\calX_{T_{n}(\beta)})_{\beta \in \calB_{n}}$ is a subnet of
  $(\calX_{T_{n-1}(\beta)})_{\beta \in \calB_{n-1}}$ and
  $\MMD[\frac{1}{n}]{\mm{f_K\cdot \nu_{T_n(\beta)}}} \geq \epsilon$
  eventually. By a diagonal argument we can construct a subnet
  $(\calX_{T(\beta)})_{\beta \in \calB}$ such that
  $\MMD{\mm{f_K \cdot \nu_\beta}} \geq \epsilon$ eventually for every
  $\delta>0$. By Corollary \ref{cor:tGwConvergenceAndMMD} this subnet
  cannot have a convergent subnet in contradiction to
  \ref{it:th:CompactNETSEquivalencies1}.

  \ref{it:th:CompactNETSEquivalencies3} $\Rightarrow$
  \ref{it:th:CompactNETSEquivalencies4}: Let $1>\epsilon>0$. The net
  $(\mass{}_*\nu_\alpha)_\alpha$ is compact and thus tight. Therefore,
  there is a $K$ such that eventually
  $\mass{\nu_\alpha} - \mass{f_K \cdot \nu_\alpha} <
  \nicefrac{\epsilon}{3}$. Then, assertion
  \ref{it:th:CompactNETSEquivalencies4} is a consequence of the
  following two claims, which we will prove later.

  \emph{Claim 1:} There is a positive integer $N_1$ and for every
  $\alpha \in \calA$ there is a bounded set
  $C_\alpha \subset X_\alpha$ with $\diam C_\alpha \leq N_1$ such that
  eventually
  \begin{equation}\label{eq:pr:th:CompactNETSEquivalencies1}
    f_K \cdot \nu_\alpha\left(\Complement\set{\mu \in \calM_f(X_\alpha) |
      \mu(\Complement C_\alpha) \leq \frac{\epsilon}{3}}\right) \leq \frac{\epsilon}{3}.
  \end{equation}

  \emph{Claim 2:} There is a positive integer $N_2$ and for every
  $\alpha \in \calA$ there is a finite set
  $A_{\alpha} \subset X_\alpha$ with $\abs{A_\alpha} \leq N_2$ such
  that eventually
  \begin{equation}\label{eq:pr:th:CompactNETSEquivalencies2}
    f_K\cdot\nu_\alpha \left(\Complement\set{\mu \in \calM_f(X_\alpha) |
      \mu(\Complement B(A_\alpha, \epsilon)) < \frac{\epsilon}{3}}\right) <
    \frac{\epsilon}{3}. 
  \end{equation}
  
  With these two claims we can define the subsets $X_{\alpha, \epsilon}
  = C_\alpha \cap B(A_\alpha, \epsilon)$ and $N_\epsilon = \max(N_1,
  N_2)$. Then eventually the set $X_{\alpha, \epsilon}$ has diameter of at
  most $N_\epsilon$ and can be covered by at most $N_\epsilon$ balls
  of radius $\epsilon$. Moreover, eventually we have
  \begin{align*}
    \nu_\alpha&\left(\Complement\set{\mu \in \calM_f(X_\alpha) |
    \mu(\Complement X_{\alpha, \epsilon}) \leq \epsilon}\right) \\
    &< \big(\mass{\nu_\alpha} - \mass{f_K \cdot \nu_\alpha}\big) + f_K \cdot \nu_\alpha\big(\Complement\set{\mu
      \in \calM_f(X_\alpha) | \mu(\Complement X_{\alpha, \epsilon}) \leq
      \epsilon}\big) \\
    &\leq \frac{\epsilon}{3} + f_K \cdot \nu_\alpha\left(\Complement\set{\mu
      \in \calM_f(X_\alpha) | \mu(\Complement C_{\alpha}) \leq
      \frac{\epsilon}{3}}\right) \\
    &\phantom{<}+ f_K \cdot \nu_\alpha\left(\Complement\set{\mu
      \in \calM_f(X_\alpha) | \mu(\Complement B(A_{\alpha}, \epsilon) <
      \frac{\epsilon}{3}}\right) \\
    &< \epsilon,
  \end{align*}
  and the assertion is proved.

  \emph{Proof of Claim 1:} The net
  $(\DD{\mm{f_K \cdot \nu_\alpha}})_\alpha$ is compact and thus tight.
  Therefore, there is an $a>0$ such that eventually
  \begin{displaymath}
    \DD{\mm{f_K \cdot \nu_\alpha}}([a, \infty)) < \frac{1}{2}
    \left(\frac{\epsilon}{3}\right)^4.
  \end{displaymath}
  Let $N_1$ be a positive integer with $2a \leq N_1$. We will
  construct bounded sets $C_\alpha \subset X_\alpha$ with
  $\diam C_\alpha < 2a \leq N_1$ such that eventually
  \eqref{eq:pr:th:CompactNETSEquivalencies1} holds. If
  \begin{displaymath}
    f_K\cdot\nu_\alpha\left(\Complement\set{\mu \in \calM_f(X_\alpha) |
      \mass{\mu} \leq \frac{\epsilon}{3}}\right) \leq \frac{\epsilon}{3},
  \end{displaymath}
  then this is satisfied for $C_\alpha \asn B(x, a)$ for any
  $x \in X_\alpha$. In the case where
  \begin{equation}\label{eq:pr:th:CompactNETSEquivalencies3}
    f_K\cdot\nu_\alpha\left(\Complement\set{\mu \in \calM_f(X_\alpha) |
      \mass{\mu} \leq \frac{\epsilon}{3}}\right) > \frac{\epsilon}{3}
  \end{equation}
  we define
  $C_\alpha \asn \set{x \in X_\alpha | \mm{f_K \cdot
      \nu_\alpha}(\Complement B(x, a)) <
    \frac{1}{2}\left(\frac{\epsilon}{2}\right)^2}$. Then,
  $\diam C_\alpha < 2a$ follows by the following contradiction: If
  there are $x, y \in C_\alpha$ with $r_\alpha(x, y) \geq 2a$ we have
  \begin{align*}
    \left(\frac{\epsilon}{2}\right)^2 &>
    \mm{f_K\cdot\nu_\alpha}(\Complement B(x, a)) +
    \mm{f_K\cdot\nu_\alpha}(\Complement B(y, a)) \\
    &\geq \mm{f_K\cdot\nu_\alpha}\big(\Complement (B(x, a) \cap B(y, a))\big) \\
    &= \mm{f_K\cdot\nu_\alpha}(X_\alpha) = \int \mu(X_\alpha) \dif(f_K
      \cdot \nu_\alpha)(\mu) \\
    &\geq \frac{\epsilon}{3}
      \Big(f_K\cdot\nu_\alpha\Big)\left(\set{\mu \in \calM_f(X_\alpha) | \mass{\mu} > \frac{\epsilon}{3}}\right).
  \end{align*}
  This contradicts \eqref{eq:pr:th:CompactNETSEquivalencies3}, so the
  diameter of $C_\alpha$ is less than $2a$.

  Furthermore, we eventually have
  \begin{align*}
    \frac{1}{2}\left(\frac{\epsilon}{3}\right)^4 &> \DD{\mm{f_K \cdot
    \nu_\alpha}}([a, \infty)) \\
    &= (\mm{f_K\cdot\nu_\alpha})^{\otimes
    2}(\set{ (x,y) \in X_\alpha^2 | y \not\in B(x,a)}) \\
    &\geq (\mm{f_K\cdot\nu_\alpha})^{\otimes
    2}(\set{ (x,y) \in X_\alpha^2 | x \not\in C_\alpha, y \not\in
      B(x,a)}) \\
    &\geq \frac{1}{2}\left(\frac{\epsilon}{3}\right)^2 \mm{f_K \cdot \nu_\alpha}(\Complement C_\alpha).
  \end{align*}
  In the last inequality we used the fact that
  \begin{displaymath}
    \mm{f_K\cdot\nu_\alpha}(\Complement B(x,a)) \geq \frac{1}{2}
    \left(\frac{\epsilon}{2}\right)^2 
  \end{displaymath}
  for $x\not\in C_\alpha$ by the very definition of $C_\alpha$. We
  conclude that eventually we have
  \begin{align*}
    \left(\frac{\epsilon}{2}\right)^2 &>
    \mm{f_K\cdot\nu_\alpha}(\Complement C_\alpha) = \int \mu(\Complement
    C_\alpha) \dif(f_K\cdot\nu_\alpha)(\mu) \\
    &\geq \frac{\epsilon}{3}
    \Big(f_K\cdot\nu_\alpha\Big)\left(\set{\mu \in \calM_f(X_\alpha) |
    \mu(\Complement C_\alpha) > \frac{\epsilon}{3}}\right)
  \end{align*}
  and this yields \eqref{eq:pr:th:CompactNETSEquivalencies1}.

  \emph{Proof of Claim 2:} Because $(\mass{}_*\nu_\alpha)_\alpha$ is a
  compact net, $\mass{\nu_\alpha}$ is eventually bounded by some
  positive real number $m$. By assumption there is a $\delta>0$ such
  that eventually
  $\MMD{\mm{f_K \cdot \nu_\alpha}} < \frac{\epsilon^2}{9}$. Set
  $N_2 \asn \max(1, \lfloor\frac{Km}{\delta}\rfloor)$. By
  Lemma~\ref{lm:MMDProperties} we can find for every
  $\alpha \in \calA$ a finite set $A_\alpha \subset X_\alpha$ with
  $\abs{A_\alpha} \leq N_2$ such that eventually
  \begin{align*}
    \frac{\epsilon^2}{9} &> \mm{f_K\cdot \nu_\alpha}(\Complement
    B(A_\alpha, \frac{\epsilon^2}{9})) \\
    &\geq \mm{f_K\cdot \nu_\alpha}(\Complement B(A_\alpha, \epsilon)) \\
    &= \int \mu(\Complement B(A_\alpha, \epsilon))
    \dif(f_K \cdot \nu_\alpha)(\mu) \\
    &\geq \frac{\epsilon}{3} \Big(f_K\cdot\nu_\alpha\Big)\left(\set{\mu \in
      \calM_f(X_\alpha) | \mu(\Complement B(A_\alpha, \epsilon)) 
      \geq \frac{\epsilon}{3} }\right).
  \end{align*}
  This leads to the desired inequality
  \eqref{eq:pr:th:CompactNETSEquivalencies2}.

  \ref{it:th:CompactNETSEquivalencies4} $\Rightarrow$
  \ref{it:th:CompactNETSEquivalencies2}: Let $\epsilon>0$ be
  arbitrary and let $N_\epsilon$ and $X_{\alpha, \epsilon}$ be as in
  assertion \ref{it:th:CompactNETSEquivalencies4}.
  $X_{\alpha, \epsilon}$ can eventually be covered by
  $N_\alpha \leq N_\epsilon$ balls
  $B(x_1^{(\alpha)}, \epsilon), \dotsc, B(x_{N_\alpha}^{(\alpha)},
  \epsilon)$. Define a function $F_\alpha \from X_\alpha \to X_\alpha$ by
  \begin{align*}
    F_\alpha(x) =
    \begin{cases}
      x_1^{(\alpha)}, &\text{ if $x \in B(x_1^{(\alpha)}, \epsilon)$ or
        $x \not\in \bigcup_{j=1}^{N_\alpha}B(x_j^{(\alpha)}, \epsilon)$} \\
      x_i^{(\alpha)}, &\text{ if $x \in B(x_i^{(\alpha)}, \epsilon) \setminus
        \bigcup_{j=1}^{i-1}B(x_j^{(\alpha)}, \epsilon)$ for $i \in
        \set{2, \dotsc, N_\alpha}$}.
    \end{cases}
  \end{align*}
  By assertion~\ref{it:lm:pushforwardproperties5} of
  Lemma~\ref{lm:pushforwardproperties} we have
  $\dP(\mu, F_{\alpha*}\mu)\leq \epsilon$ for every
  $\mu \in \calM_f(X_\alpha)$ with
  $\mu(\Complement X_{\alpha, \epsilon}) \leq \epsilon$. Thus,
  eventually we have
  $\dP(\nu_\alpha, F_{\alpha**}\nu_\alpha) \leq \epsilon$.

  Because $(\mass{}_*\nu_\alpha)_\alpha$ is compact,
  $\mass{\nu_\alpha}$ is eventually bounded from above by some
  positive number $m$ and $(\mass{}_*\nu_\alpha)_\alpha$ is tight.
  Therefore, there is a $K>0$ such that eventually
  $\dP(\nu_\alpha', \nu_\alpha) < \epsilon$, where $\nu_\alpha'$ is
  the restriction of $\nu_\alpha$ to $\calM_{\leq K}(X_\alpha)$. By
  the triangle inequality we eventually have
  $\dP(\nu_\alpha, F_{\alpha**}\nu_\alpha') \leq 2\epsilon$. Moreover,
  $(X_\alpha, r_\alpha, F_{\alpha**}\nu'_\alpha)$ is eventually in $\A_N$ for
  $N \asn \max(m, K, N_\epsilon)$. This proves the claim since
  $\epsilon$ was arbitrary.
\end{proof}

\section{Equivalence of both topologies}
\label{sec:EquivalenceOfBothTopologies}

We are finally able to prove the fact that the two-level Gromov-weak
topology and the two-level Gromov-Prokhorov topology on $\Mtwo$
coincide. Our proof relies mostly on the compactness criteria of
Theorem~\ref{th:CompactNETSEquivalencies}. 

\begin{theorem}\label{th:dtGPandtGwConvergenceCoincide}
  The two-level Gromov-weak topology and the two-level Gromov-Prokhorov
  topology coincide.
\end{theorem}

\begin{proof}
  We will show that the identity function
  $\id \from (\Mtwo, \tautGw) \to (\Mtwo, \tautGP)$ and its inverse
  $\id^{-1}$ are both continuous. Because $\tautGP$ stems from a
  metric, $\id^{-1}$ is continuous if and only if it is sequentially
  continuous and the latter was proved
  in Lemma~\ref{lm:d2GPconvergenceImpliesWeakConvergence}.
  However, we do not know yet if $\tautGw$ is metrizable (nor if it is
  first countable). Therefore, it is not enough to show that $\id$ is
  sequentially continuous. Instead we have to work with nets and show
  that each $\tautGw$-convergent net
  $\calX_\alpha = (X_\alpha, r_\alpha, \nu_\alpha) \tGwto \calX = (X,
  r, \nu)$ also converges in the two-level Gromov-Prokhorov topology.
  We will do so by proving that $(\calX_\alpha)_\alpha$ is a compact
  net with respect to $\tautGP$. Because $\tfset$ separates points
  (\cf Theorem~\ref{th:ReconstructionTheorem}), this implies that
  every subnet has a converging subnet with the same limit $\calX$.
  Hence, $(\calX_\alpha)_\alpha$ itself converges to $\calX$ in the
  $\dtGP$-metric.

  To prove compactness, we use
  property~\ref{it:th:CompactNETSEquivalencies3} of
  Theorem~\ref{th:CompactNETSEquivalencies}. First of all, the
  functions $(X, r, \nu) \mapsto \mass{}_*\nu$ and
  $(X, r, \nu) \mapsto \DD{\mm{f_K\cdot\nu}}$ are continuous with
  respect to $\tautGw$ by
  Remark~\ref{rm:MassIs2lvlGromovWeakContinuous} and
  Lemma~\ref{lm:fKisContinuous}. Hence, both
  $(\mass{}_*\nu_\alpha)_\alpha$ and
  $(\DD{\mm{f_K\cdot\nu_\alpha}})_\alpha$ are compact. Furthermore, by
  Corollary~\ref{cor:tGwConvergenceAndMMD} for every $\epsilon>0$
  there is a $\delta>0$ such that eventually
  $\MMD{\mm{f_K \cdot \nu_\alpha}} < \epsilon$. Thus, the assumptions
  of Theorem~\ref{th:CompactNETSEquivalencies} are satisfied and
  $(\calX_\alpha)_\alpha$ is a compact net in the two-level
  Gromov-Prokhorov topology.
\end{proof}

Note that now every statement we made about one of the two topologies
(\eg embedding theorems, compactness criteria, etc.) also holds true
for the other topology.

\section{Tightness and convergence determining sets}
\label{sec:tightness}

In this section we show that the set of test functions $\tfset$ is
convergence determining for $\calM_1(\Mtwo)$ and characterize tight
subsets of $\calM_1(\Mtwo)$. Let us first recall the definition of
convergence determining sets of functions.

\begin{definition}[Convergence determining sets]
  Let $(X, r)$ be a metric
  space. 
  A set $\calF \subset \calC_b(X)$ is called \emph{convergence
    determining} for $\calM_1(X)$ if for
  $\mu, \mu_1, \mu_2, \dotsc \in \calM_1(X)$ weak convergence
  $\mu_n \wto \mu$ is equivalent to
  \begin{align*}
    \int f \dif\mu_n \to \int f \dif\mu \quad \forall f \in \calF.
  \end{align*}
\end{definition}

\begin{theorem}
  The set of test functions $\tfset$ is convergence determining for
  $\calM_1(\Mtwo)$.
\end{theorem}
\begin{proof}
  The set $\tfset$ separates points, is closed under multiplication
  and induces the topology of $\Mtwo$. Therefore, it is convergence
  determining for $\calM_1(\Mtwo)$ by
  \cite[Lemma 4.1]{Hoffmann77}.
\end{proof}

This means that a sequence $(P_n)_n$ in $\calM_1(\Mtwo)$ converges to
$P \in \calM_1(\Mtwo)$ if and only if $P_n[\Phi]$ converges to
$P[\Phi]$ for every $\Phi \in \tfset$. Here, $P[\Phi]$ denotes the
expectation $\int \Phi(\calX) \dif P(\calX)$.

We now give a characterization of tight subsets of $\calM_1(\Mtwo)$.
Since tightness is defined in terms of compact sets, it is not
surprising that we use Theorem~\ref{th:CompactnessEquivalencies} to
find conditions for tightness.
\begin{proposition}[Characterization of tight sets]\label{prp:MtwoTightness}
  A set $\calP \subset \calM_1(\Mtwo)$ is tight if and only if for
  every $\epsilon>0$ and $K>0$ there are $\delta>0$ and $c>0$ such
  that for every $P \in \calP$ we have
  \begin{enumerate}
    \item $P(\mass{\nu} \geq c) < \epsilon$,
    \item $P(\mass{}_*\nu([c, \infty)) \geq \epsilon) < \epsilon$,
    \item $P(\MMD{\mm{f_K \cdot \nu}} \geq \epsilon) < \epsilon$,
    \item $P(\DD{\mm{f_K \cdot \nu}}([c, \infty)) \geq \epsilon ) <
    \epsilon$.
  \end{enumerate}
\end{proposition}

\begin{proof}
  Let $\epsilon, K > 0$. If $\calP$ is tight, there is a compact set
  $C \subset \Mtwo$ with $P(\Complement C) < \epsilon$ for every
  $P \in \calP$. By property~\ref{it:th:CompactnessEquivalencies4} of
  Theorem~\ref{th:CompactnessEquivalencies} we can choose $\delta>0$
  and $c> \sup\set{\mass{\nu} | (X, r, \nu) \in C}$ such that for
  every $(X, r, \nu) \in C$ we have
  \begin{align*}
   &\mass{}_*\nu([c, \infty)) < \epsilon, \\
    &\MMD{\mm{f_K \cdot \nu}} < \epsilon,  \\
    &\DD{\mm{f_K \cdot \nu}}([c, \infty)) < \epsilon
  \end{align*}
  and the claim follows immediately.

  To prove the other direction, let $\epsilon>0$. We are going to
  construct a relatively compact set $C \subset \Mtwo$ such that
  $P(\Complement C) < \epsilon$ for every $P \in \calP$. First, define
  $\epsilon_n \asn \frac{\epsilon}{4} \cdot 2^{-n}$ and $K_n \asn n$
  for every $n \in \N$. By assumption there are $c, \delta_n, c_n > 0$
  such that
  \begin{align*}
    &P(\mass{\nu} \geq c) < \tfrac{\epsilon}{4}, \\
    &P(\mass{}_*\nu([c_n, \infty)) \geq \epsilon_n) < \epsilon_n, \\
    &P( \MMD[\delta_n]{\mm{f_{K_n} \cdot \nu}} \geq \epsilon_n ) < \epsilon_n,\\
    &P( \DD{\mm{f_{K_n} \cdot \nu}}([c_n, \infty)) \geq \epsilon_n) < \epsilon_n
  \end{align*}
  for every $P \in \calP$ and $n \in \N$. Let $C_n$ be the set of
  all $(X,r,\nu) \in \Mtwo$ with
  \begin{align*}
    &\mass{}_*\nu([c_n, \infty)) < \epsilon_n,\\
    &\MMD[\delta_n]{\mm{f_{K_n} \cdot \nu}} < \epsilon_n,\\
    &\DD{\mm{f_{K_n} \cdot \nu}}([c_n, \infty)) < \epsilon_n.
  \end{align*}
  We have $P(\Complement C_n) < 3\epsilon_n$ for every
  $P \in \calP$. With
  \begin{align*}
    C \asn \set{ (X,r,\nu) \in
    \Mtwo | \mass{\nu} < c} \cap \bigcap_{n \in \N}C_n,
  \end{align*}
  we get
  \begin{align*}
    P(\Complement C) < \frac{\epsilon}{4} + \sum_{n \in
    \N}3\epsilon_n = \epsilon
  \end{align*} 
  for every $P \in \calP$. Moreover, $C$ satisfies the compactness
  criterion given in property~\ref{it:th:CompactnessEquivalencies4} of
  Theorem~\ref{th:CompactnessEquivalencies} and thus is relatively
  compact.
\end{proof}

If an m2m space $(X, r, \nu)$ satisfies $\nu \in \calM_1(\calM_1(X))$,
its first moment measure $\mm{\nu}$ is always finite and we do not
need to approximate it by $\mm{f_K \cdot \nu}$. Let us define the
subset of all m2m spaces with this property. 
\begin{definition}
  We define the set of \emph{metric two-level probability measure spaces}
  $\MtwoProb$ by
  \begin{align*}
    \MtwoProb \asn \set{(X,r,\nu) \in \Mtwo | \nu \in \calM_1(\calM_1(X))}.
  \end{align*}
  $\MtwoProb$ is a closed subset of $\Mtwo$ and thus a Polish metric
  space with the restriction of the metric $\dtGP$.
\end{definition}

Proposition~\ref{prp:MtwoTightness} boils down to a much simpler
version if $P(\MtwoProb) = 1$ for all $P \in \calP$. We omit
the obvious proof.
\begin{corollary}\label{cor:MtwoProbTightness1}
  A set $\calP \subset \calM_1(\MtwoProb)$ is tight if and only if for
  every $\epsilon>0$ there are $\delta>0$ and $c>0$ such that for
  every $P \in \calP$ we have
  \begin{enumerate}
    \item\label{it:cor:MtwoProbTightness1_2}
    $P(\DD{\mm{\nu}}([c, \infty)) \geq \epsilon ) < \epsilon$ and
    \item\label{it:cor:MtwoProbTightness1_1} $P(\MMD{\mm{\nu}} \geq \epsilon) < \epsilon$.
  \end{enumerate}
\end{corollary}

The following version of the preceding corollary is particularly
useful for the application in Section~\ref{sec:NestedKingman}.
\begin{corollary}\label{cor:MtwoProbTightness2}
  A set $\calP \subset \calM_1(\MtwoProb)$ is tight if the following
  two conditions hold:
  \begin{enumerate}
    \item\label{it:cor:MtwoProbTightness2_2}
    There is a finite Borel measure $\mu$ on $\Rplus$, such that
    \begin{displaymath}
      P[\DD{\mm{\nu}}] = \int \DD{\mm{\nu}} \dif P((X, r, \nu)) \leq \mu
    \end{displaymath}
    for every $P \in \calP$.
    \item\label{it:cor:MtwoProbTightness2_1}
    $ \lim\limits_{\delta \searrow 0} \sup\limits_{P \in \calP}
    P[\mm{\nu}(\set{x \in X | \mm{\nu}(\clBall(x, \epsilon)) < \delta})]
    = 0$ for every $\epsilon>0$.
  \end{enumerate}
\end{corollary}
\begin{remark}
  In case $\calP$ is a sequence $(P_n)_n$, we can replace
  $\sup\limits_{P \in \calP}$ in
  property~\ref{it:cor:MtwoProbTightness2_1} by
  $\limsup\limits_{n \to \infty}$.
\end{remark}
\begin{proof}[of Corollary~\ref{cor:MtwoProbTightness2}]
  We show that $\calP$ satisfies both properties of
  Corollary~\ref{cor:MtwoProbTightness1}. Let $\epsilon>0$. There
  exists a $c>0$ such that $\mu([c, \infty)) < \epsilon^2$. By
  Markov's inequality we get
  \begin{align*}
    P(\DD{\mm{\nu}}([c, \infty)) \geq \epsilon) \leq
    \frac{P[\DD{\mm{\nu}}([c, \infty))]}{\epsilon} \leq \frac{\mu([c,
    \infty))}{\epsilon} < \epsilon
  \end{align*}
  for every $P \in \calP$. Moreover, by
  condition~\ref{it:cor:MtwoProbTightness2_1} there is a $\delta>0$
  such that
  \begin{align*}
    P[\mm{\nu}(\set{x \in X | \mm{\nu}(\clBall(x, \frac{\epsilon}{2}))
    < 2\delta})] < \epsilon^2
  \end{align*}
  for all $P \in \calP$. With Lemma~\ref{lm:MMDProperties} and with
  Markov's inequality we get
  \begin{align*}
    P(\MMD{\mm{\nu}} \geq \epsilon) &= P(\mm{\nu}(\set{x \in X |
    \mm{\nu}(B(x, \epsilon)) \leq \delta}) \geq \epsilon) \\ &\leq
    \epsilon^{-1} P[\mm{\nu}(\set{x \in X | \mm{\nu}(B(x, \epsilon))
    \leq \delta})] \\ &\leq
    \epsilon^{-1} P[\mm{\nu}(\set{x \in X | \mm{\nu}(\clBall(x,
                        \frac{\epsilon}{2})) < 2\delta})] \\
    &< \epsilon
  \end{align*}
  for all $P \in \calP$. Thus, $\calP$ is tight by
  Corollary~\ref{cor:MtwoProbTightness1}.
\end{proof}

\section{Example: The nested Kingman coalescent measure tree}
\label{sec:NestedKingman}

In this section we define the finite and infinite nested Kingman
coalescent. Moreover, we introduce the nested Kingman coalescent
measure tree, which is a random m2m space defined as the weak limit of
finite random m2m spaces. To show convergence we will apply the
tightness criteria from Section~\ref{sec:tightness}.

\subsection{The nested Kingman coalescent}

Nested coalescents were introduced in \cite{AiramThesis, Blancas18a}
to jointly model the species and the gene coalescents of a population
of multiple species. The nested Kingman coalescent is a
special case of the model developed in these publications (\cf also
\cite{Blancas18b} for further research about the nested Kingman
coalescent). In this subsection we give a definition of the nested
Kingman coalescent first for a finite set $I \subset \N^2$ of
individuals and then for infinitely many individuals (\ie $I = \N^2$).
We use $\N^2$ to encode individuals because we think of
$(i, j) \in \N^2$ as the $j$-th individual of the $i$-th species.

For a non-empty set $I$ let $\calE(I) \subset I^2$ denote the set of
equivalence relations on $I$ equipped with the discrete topology. The
equivalence classes of an equivalence relation are called
\emph{blocks}. We say that a pair
$(\calR_1, \calR_2) \in \calE(I) \times \calE(I)$ is \emph{nested} (or
that $\calR_2$ is nested in $\calR_1$) if $\calR_1 \supset \calR_2$.
Note that $(\calR_1, \calR_2)$ is nested if and only if for every
block $\pi_2$ of $\calR_2$ there is a block $\pi_1$ of $\calR_1$ with
$\pi_2 \subset \pi_1$. Let
\begin{displaymath}
\calN(I) \asn \set{(\calR_1, \calR_2) \in \calE(I)^2 | \calR_2
  \subset \calR_1}
\end{displaymath}
denote the set of nested equivalence relations
equipped with the discrete topology.

Moreover, we define the following equivalence relations on $\N^2$,
which will be the initial states of the nested Kingman coalescent:
\begin{align*}
  G_0 &\asn \set{ (\vect{x}, \vect{x}) | \vect{x} \in \N^2}, \\
  S_0 &\asn \set{ ((i,j), (i,k)) | i, j, k \in \N}.
\end{align*}
$G_0$ is the equivalence relation with only singleton blocks and $S_0$
is the equivalence relation whose blocks are the different species of
the population.

\begin{definition}[Finite nested Kingman coalescent]
  Let $I$ be a finite subset of $\N^2$ and $\gamma_s, \gamma_g > 0$.
  Let
  $\calR^{(I)} = (\calR^{(I)}(t))_{t \geq 0} = (\calR_s^{(I)}(t),
  \calR_g^{(I)}(t))_{t\geq 0} $ be a continuous-time Markov process with
  values in $\calN(I)$. We call $\calR^{(I)}$ the \emph{finite nested
    Kingman coalescent on $I$ with rates $(\gamma_s, \gamma_g)$} if it
  has the following properties:
  \begin{itemize}
    \item The initial state is $\calR_s^{(I)}(0) = S_0 \cap I^2$ and
    $\calR_g^{(I)}(0) = G_0 \cap I^2$, \ie for
    $\vect{x} = (x_1, x_2) \in \I^2$ and
    $\vect{y} = (y_1, y_2) \in I^2$ we have
    \begin{align*}
      (\vect{x}, \vect{y}) \in \calR_s^{(I)}(0) &\Leftrightarrow
      x_1 = y_1, \\
      (\vect{x}, \vect{y}) \in \calR_g^{(I)}(0) &\Leftrightarrow
      \vect{x} = \vect{y}.
    \end{align*}
    \item The \emph{species coalescent}
    $\calR_s^{(I)} = (\calR_s^{(I)}(t))_{t \geq 0}$ behaves like a
    Kingman coalescent with rate $\gamma_s$, \ie any two blocks in
    $\calR_s^{(I)}(t)$ merge at rate $\gamma_s$
    \item The \emph{gene coalescent}
    $\calR_g^{(I)} = (\calR_s^{(I)}(t))_{t \geq 0}$ behaves in the
    following way: any two blocks $\pi_1, \pi_2$ of $\calR_g^{(I)}$
    such that $\pi_1 \cup \pi_2$ is contained in a single block of
    $\calR_s^{(I)}(t)$ merge at rate $\gamma_g$. Other blocks cannot
    merge.
  \end{itemize}
\end{definition}

\begin{figure}
  \centering
  \includegraphics{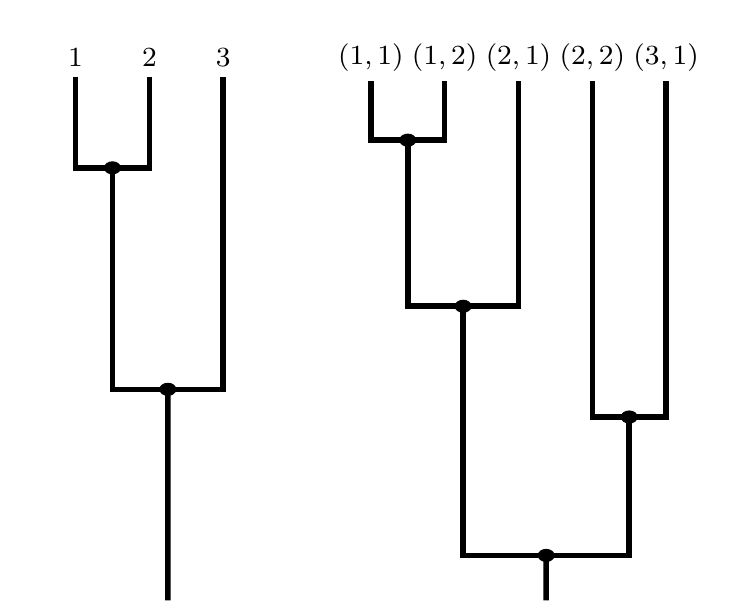}
  \caption{A realization of a finite nested Kingman coalescent on the
    set $I = \set{(1,1), (1,2), (2,1), (2,2), (3,1)}$. The species
    coalescent on the left starts with the equivalence classes
    $\set{(1,1), (1,2)}, \set{(2,1), (2,2)}$ and $\set{(3,1)}$ (here
    abbreviated by 1, 2 and 3). Its branch points are speciation
    events. The tree on the right side depicts the gene coalescent.
    Notice that merging events in the gene coalescent can only happen
    after the corresponding species have merged in the species
    coalescent.}
  \label{fig:FiniteNestedKingman}
\end{figure}

The definition of the finite nested Kingman coalescent describes the
behavior of a Markov process with only finitely many states. Thus, it
is clear that such a process exists and is unique in distribution.
Figure~\ref{fig:FiniteNestedKingman} shows a realization of a finite
nested Kingman coalescent.

We now define the nested Kingman coalescent for an infinite set of
individuals (that is, with $I = \N^2$). For the existence of this
process we refer to the construction of (more general) nested
coalescents in \cite[section 5]{Blancas18a}.
\begin{definition}[Nested Kingman coalescent]
  Let $\gamma_s, \gamma_g > 0$. The \emph{nested Kingman coalescent
    with rates $(\gamma_s, \gamma_g)$} is a continuous-time Markov
  process
  $\calR = (\calR(t))_{t \geq 0} = (\calR_s(t), \calR_g(t))_{t \geq
    0}$ with values in $\calN(\N^2)$ such that for any finite
  $I \subset \N^2$ the restriction of $\calR$ to $\calN(I)$ is a
  finite nested Kingman coalescent on $I$ with rates
  $(\gamma_s, \gamma_g)$.
\end{definition}

It follows immediately from the definition that the initial states of
the species and the gene coalescent are
\begin{align*}
  \calR_s(0) = S_0 \quad \text{and} \quad \calR_g(0) = G_0.
\end{align*}

It is well-known that the standard Kingman coalescent immediately
comes down from infinity, meaning that after any positive time the
coalescent almost surely has only finitely many blocks left, even if
we start with infinitely many blocks. The same is true for the nested
Kingman coalescent as stated in the next lemma. A proof can be found
in \cite[section 6]{Blancas18a}.
\begin{lemma}\label{lm:NestedKingmanComesDownFromInfty}
  The nested Kingman coalescent immediately comes down from infinity.
  That is, if $\calR = (\calR_s(t), \calR_g(t))_{t \geq 0}$ is a
  nested Kingman coalescent, then for every $t>0$ both $\calR_s(t)$
  and $\calR_g(t)$ almost surely consist of only finitely many blocks.
\end{lemma}

\subsection{The nested Kingman coalescent measure tree}

In this subsection we define a random m2m space called the nested
Kingman coalescent measure tree. Roughly speaking, it is the
genealogical tree of the gene coalescent of a nested Kingman
coalescent equipped with a two-level measure that represents uniform
sampling of species on the second level and uniform sampling of
individuals in a single species on the first level.

Let $\calR = (\calR_s, \calR_g)$ be a nested Kingman coalescent and
$\P$ its law. For $\vect{x} = (x_1, x_2) \in \N^2$ and
$\vect{y} = (y_1, y_2) \in \N^2$ we define the coalescence time of
$\vect{x}$ and $\vect{y}$ by
\begin{align*}
  r_g(\vect{x}, \vect{y}) \asn \inf\set{t \geq 0 | (\vect{x},
  \vect{y}) \in \calR_g(t)}.
\end{align*}
The law of $r_g(\vect{x}, \vect{y})$ is
\begin{equation}\label{eq:NestedKingmanDistrOfDistances}
  \begin{aligned}
    r_g(\vect{x}, \vect{y}) &\sim
    \begin{cases}
      \ExpDistribution{\gamma_g}  \ast \ExpDistribution{\gamma_s}, & \text{
        if } x_1
      \not= y_1  \\
      \ExpDistribution{\gamma_g}, & \text{ if $x_1 = y_1$ and $x_2\not=
        y_2$} \\
      \delta_0, & \text{ if $\vect{x}=\vect{y}$},
    \end{cases} 
  \end{aligned}
\end{equation}
where $\ExpDistribution{\gamma}$ denotes the exponential distribution
with parameter $\gamma>0$ and $\mu\ast\eta$ denotes the convolution of
two distributions $\mu$ and $\eta$. The function $r_g$ satisfies the
triangle inequality and is thus a (random) metric on $\N^2$ (in fact
it is even an ultra-metric). Let $(Z, r)$ denote the completion of the
metric space $(\N^2, r_g)$.

\begin{remark}[Partial exchangeability]
  \label{rm:NestedKingmanIsPartExchangeable}
  The distances between individuals of the nested Kingman coalescent
  are partially exchangeable in the following sense. Let
  $\widetilde{P}$ be the set of finite permutations $p$ on $\N^2$ with
  the property that $\pi_1 p(i, j) = \pi_1 p(i,k)$ for all
  $i, j, k \in \N$, where $\pi_1$ denotes the projection to the first
  component of a vector (\ie the first components of $p(\vect{x})$ and
  $p(\vect{y})$ coincide whenever the first components of
  $\vect{x}, \vect{y} \in \N^2$ coincide).
  Then, for every $\vect{x_1}, \dotsc, \vect{x_n} \in \N^2$ the law of
  $R(\vect{x_1}, \dotsc, \vect{x_n})$ is equal to the law of
  $R(p(\vect{x_1}), \dotsc, p(\vect{x_n}))$ for every permutation
  $p \in \widetilde{P}$. In other words, the law of
  $R(\vect{x_1}, \dotsc, \vect{x_n})$ is invariant under
  $\widetilde{P}$.
\end{remark}

For all $M, N \in \N$ we define the two-level measure
$\nu_{M, N} \in \calM_1(\calM_1(Z))$ by
\begin{equation}\label{eq:NuMN}
  \nu_{M, N} \asn \frac{1}{M} \sum_{i=1}^M \delta_{\left(\frac{1}{N}
  \sum_{j=1}^N \delta_{(i, j)}\right)}.
\end{equation}
$\nu_{M, N}$ samples uniformly one of the first $M$ species and then
samples uniformly one of the first $N$ individuals in that species.
Let $H_{M, N}$ be the function that maps a realization of the nested
Kingman coalescent to the m2m space $(Z, r, \nu_{M, N})$ and define
\begin{equation*}
  \Q_{M, N} \asn H_{M,N*}\P \in \calM_1(\Mtwo).
\end{equation*}

\begin{theorem}\label{th:nestedKingmanM2M}
  \begin{enumerate}
    \item\label{it:th:nestedKingmanM2M1} The sequence $(\Q_{M, N})_N$
    is weakly convergent for every $M \in \N$. We denote its limit by
    $\Q_M$.
    \item\label{it:th:nestedKingmanM2M2} The sequence $(\Q_{M})_M$ is
    weakly convergent. 
  \end{enumerate}
\end{theorem}
Therefore, the limit
\begin{displaymath}
\Q \asn \wlim_{M\to\infty}\wlim_{N\to\infty}\Q_{M, N} = \wlim_{M\to\infty}\Q_M
\end{displaymath}
exists. We call any random variable with values in $\Mtwo$ and law
$\Q$ a \emph{nested Kingman coalescent measure tree}.

\begin{remark}[Generalization of Theorem~\ref{th:nestedKingmanM2M}]
  Recall that the $\Lambda$-coalescent is a generalization of the
  Kingman coalescent that allows multiple mergers and in which the
  merging rates are described by a finite measure
  $\Lambda \in \calM_f([0, 1])$. In a similar manner we can generalize
  the nested Kingman coalescent to a \emph{nested
    $(\Lambda_s, \Lambda_g)$-coalescent}, where the species coalescent
  behaves like a $\Lambda_s$-coalescent and the gene coalescent
  behaves like a $\Lambda_g$-coalescent (inside of single species
  blocks). The proof of Theorem~\ref{th:nestedKingmanM2M} is valid
  even for the nested $(\Lambda_s, \Lambda_g)$-coalescent as long as
  the nested coalescent immediately comes down from infinity. The
  latter condition is true if and only if both $\Lambda_s$ and
  $\Lambda_g$ are such that the corresponding $\Lambda$-coalescents
  immediately come down from infinity (\cf \cite{Blancas18b}).

  {
    It is an open question whether Theorem~\ref{th:nestedKingmanM2M}
    still holds for $\Lambda$-coalescents which are dust-free but do
    not come down from infinity (\cf \cite[Theorem 4]{GPW09} in which
    the authors construct a $\Lambda$-coalescent measure tree for all
    $\Lambda$-coalescents which satisfy the dust free property). We
    strongly suspect that this is true. However, a necessary
    ingredient for our proof (in particular
    Lemma~\ref{lm:AllFrequenciesAre0}) is that the nested coalescent
    immediately comes down from infinity and we were not able to find a
    more general proof.
  }
\end{remark}

\subsection{Proof of Theorem~\ref{th:nestedKingmanM2M}}

The proofs of both statements of Theorem~\ref{th:nestedKingmanM2M}
use the same kind of argument. First we show that the sequence under
consideration has at most one limit point. Then we show that the
sequence is relatively compact, \ie every subsequence has a convergent
subsequence. Consequently, because the limit point is unique, we may
conclude that the original sequence is convergent.

Two main tools when working with coalescents are exchangeability and
relative frequencies of blocks. We already showed in
Remark~\ref{rm:NestedKingmanIsPartExchangeable} that the nested
Kingman coalescent is partially exchangeable. Let us now define the
relative frequencies of blocks.
\begin{definition}
  For all $i, l \in \N$ and $t \in \Rplus$ we define 
  \begin{align*}
   \freq{i}{l}{t} &\asn \lim_{N \to \infty}\frac{1}{N}
                   \sum_{n=1}^N 
                   \One_{[0,t]}(r_g( (i,1), (l, n))).
  \end{align*}
  $\freq{i}{l}{t}$ is the relative frequency of the block of $(i, 1)$
  \wrt species $l$ at time $t$. 

\end{definition}

It is a standard fact from coalescent theory that in a Kingman
coalescent the relative frequencies of blocks exist. By definition the
gene coalescent restricted to a single species $i \in \N$ is a Kingman
coalescent. Thus, the relative frequency $\freq{i}{i}{t}$ exists for
every $t \geq 0$. Moreover, the Kingman coalescent almost surely has
proper frequencies (\cf \cite[Theorem 8]{Pitman99}), implying that
\begin{equation}\label{eq:th:nestedKingmanM2M_0}
  \P(\freq{i}{i}{t}=0) = 0
\end{equation}
for every $t > 0$.

For $\freq{i}{l}{t}$ with $i \not= l$ the situation is a little
different since the species $i$ and $l$ have to merge first. Let
$\tau_s(i, l)$ denote the coalescent time of the species $i$ and $l$
in the species tree. Clearly, we have $\freq{i}{l}{t} = 0$ for
$t < \tau_s(i, l)$. For $t \geq \tau_s(i, l)$ the gene coalescent
restricted to $\set{ (l, j) | j \in \N}\cup \set{(i, 1)}$ behaves like
a Kingman coalescent. Thus, the relative frequency $\freq{i}{l}{t}$
also exist for $t \geq \tau_s(i, l)$.

Because the nested Kingman coalescent starts with singleton blocks at
time $t=0$, we have $\freq{i}{l}{0} = 0$ for all $i, l \in \N$.
Moreover, almost surely the function $t \mapsto \freq{i}{l}{t}$ is
non-decreasing (the relative frequency increases when blocks merge)
and converges to 1 (eventually all blocks have merged to a single
block).

In the next lemma we state an analogon of
equation~\eqref{eq:th:nestedKingmanM2M_0} for the relative frequencies
$\freq{i}{l}{t}$ with $i \not= l$. Though it may happen that
$\freq{i}{l}{t} = 0$ \emph{for some} $l \not= i$, it may not happen
\emph{for all} $l \not= i$. This follows from the partial
exchangeability explained in
Remark~\ref{rm:NestedKingmanIsPartExchangeable}.
\begin{lemma}\label{lm:AllFrequenciesAre0}
  For every $t>0$ and every $i \in \N$
  \begin{equation}\label{eq:lm:AllFrequenciesAre0_1}
    \P(\freq{i}{l}{t} = 0 \text{ for all }l \not= i) = 0.
  \end{equation}
\end{lemma}
\begin{proof}
  We define for every $n \in \N$ and $t' \in \Rplus$ the infinite
  vector $\freqvect{n}{t'} \in [0, 1]^{\N}$ by
  \begin{align*}
    \freqvect{n}{t'} \asn (\freq{n}{l}{t'})_{l \not= n} =
    (\freq{n}{1}{t'}, \freq{n}{2}{t'}, \dotsc,
    \freq{n}{n-1}{t'} , \freq{n}{n+1}{t'}, \dotsc).
  \end{align*}
  We fix $t > 0$ and $i \in \N$. Observe that
  equation~\eqref{eq:lm:AllFrequenciesAre0_1} is equivalent to
  $\P(\freqvect{i}{t} = \vect{0}) = 0$ and that the sequence of
  vectors $(\freqvect{n}{t})_{n \in \N}$ is exchangeable (\cf
  Remark~\ref{rm:NestedKingmanIsPartExchangeable}). By de Finetti's
  theorem there is a random probability measure $\Xi$ on $[0, 1]^\N$
  such that $\Xi^{\otimes \N}$ is a regular conditional distribution
  of $(\freqvect{n}{t})_{n}$ given $\sigma(\Xi)$ (\cf \cite[Theorem
  3.1]{Aldous85}).

  We prove the claim by contradiction. Assume that
  \begin{align*}
    0 < \P(\freqvect{i}{t} = \vect{0}) = \int \Xi(\vect{0}) \dif\P.
  \end{align*}
  This is true if and only if $\P(\Xi(\vect{0})>0) > 0$ and in this
  case
  \begin{equation}\label{eq:lm:AllFrequenciesAre0_2}
    \begin{aligned}
      \P&(\freqvect{n}{t} = \vect{0} \text{ for infinitely many $n \in
        \N$}) \\
      &= \P(\freqvect{n}{t} = \vect{0} \text{ for infinitely many
        $n \in
        \N$} \mid \Xi(\vect{0}) > 0) \cdot \P(\Xi(\vect{0}) > 0) \\
      &= \P(\Xi(\vect{0}) > 0) \\
      &> 0.
    \end{aligned}
  \end{equation}
  However, if $\freqvect{n}{t} = \vect{0}$ for infinitely many
  $n \in \N$, then there is an increasing sequence of positive
  integers $(n_k)_k$ with $\freqvect{n_k}{t} = \vect{0}$. Observe that
  $\freqvect{n_k}{t} = \vect{0}$ implies that at time $t$ the block of
  $(n_k, 1)$ has not merged with a block of another species.
  Therefore, each $(n_l, 1)$ is in a different block and $\calR_g(t)$
  contains an infinite number of blocks. But by
  Lemma~\ref{lm:NestedKingmanComesDownFromInfty} we know that almost
  surely $\calR_g(t)$ contains only finitely many blocks. This is a
  contradiction to~\eqref{eq:lm:AllFrequenciesAre0_2}. Consequently,
  we must have $\P(\freqvect{i}{t} = \vect{0}) = 0$.
\end{proof}

Before we start to prove Theorem~\ref{th:nestedKingmanM2M}, observe
that the two-level measure $\nu_{M, N}$ from~\eqref{eq:NuMN}  has the
first moment measure
\begin{align*}
  \mm{\nu_{M, N}} = \frac{1}{M} \sum_{i=1}^M\frac{1}{N}
  \sum_{j=1}^N \delta_{(i, j)},
\end{align*}
\ie $\mm{\nu_{M, N}}$ is a uniform distribution on the set
$\set{1, \dotsc, M} \times \set{1, \dotsc, N}$.

\subsubsection{
Weak convergence of $(\Q_{M, N})_N$ for fixed $M$}

Let $M \in \N$ be fixed.

\paragraph{Uniqueness:}

Recall that the test functions from $\tfset$ are convergence
determining for $\calM_1(\Mtwo)$. Because all our m2m spaces are
elements of $\MtwoProb$ (\ie the measures on both levels have mass 1),
it is enough to consider test functions $\Phi$ of the form
\begin{equation}\label{eq:th:nestedKingmanM2M_1}
  \Phi((X, r, \nu)) = \int \int \phi \circ R \dif\vect{\mu}^{\otimes\vect{n}}
  \dif\nu^{\otimes m}(\vect{\mu})
\end{equation}
with $m \in \N$, $\vect{n} = (n_1, \dotsc, n_m) \in \N^m$ and
$\phi \in \calC_b(\Rplus^{\abs{\vect{n}} \times \abs{\vect{n}}})$. We
will explain why for each such $\Phi$ the limit
$\lim_{N\to\infty}\Q_{M, N}[\Phi]$ exists. Thus, the sequence
$(\Q_{M, N})_N$ has at most one limit point.

Fix a test function $\Phi$ as in \eqref{eq:th:nestedKingmanM2M_1}.
Then, we have
\begin{align*}
  \Q_{M, N}[\Phi] = \int \Phi \dif\Q_{M, N} = \int
  \Phi((Z, r, \nu_{M, N})) \dif\P = \int \int\int \phi \circ R
  \dif\vect{\mu}^{\otimes \vect{n}} \dif(\nu_{M, N})^{\otimes
  m}(\vect{\mu}) \dif\P.
\end{align*}
One can show that this converges to
\begin{equation}\label{eq:th:nestedKingmanM2M_2}
   \int \frac{1}{M^m}\sum\limits_{i_1,
  \dotsc, i_m=1}^M \phi\big(R\big({\scriptstyle(i_1, 1), \dotsc, (i_1, n_1), (i_2,
  n_1+1), \dotsc, (i_2, n_1+n_2), (i_3, n_1+n_2+1), \dotsc, (i_m, \abs{\vect{n}})}\big)\big) \dif\P 
\end{equation}
for $N \to \infty$ using the partial exchangeability of the distances
under $\P$ (\cf Remark~\ref{rm:NestedKingmanIsPartExchangeable}).
However, writing down a formal proof for general $m$ and $\vect{n}$ is
cumbersome and we would have to introduce a lot of notation. For this
reason we omit the proof. The reader may easily verify our claim for
small $m$ and $\vect{n}$ to understand what is going on here.
Heuristically, $\Phi$ corresponds to sampling $m$ species, then
sampling $n_1, \dotsc, n_m$ individuals in these species and then
evaluating the (genetic) distances between the individuals. We sample
with the two-level measure $\nu_{M, N}$, which means we uniformly
sample from the first $M$ species and in each of these species we
uniformly sample from the first $N$ individuals. Since $M$ is finite,
it is possible that some species are sampled more than once. But for
$N \to \infty$ the probability to sample a single individual more than
once goes to 0.

\paragraph{Relative compactness:}

We use Corollary~\ref{cor:MtwoProbTightness2}. Thus, we have to show
the following:
\begin{enumerate}
  \item there is a finite Borel measure $\mu_0$ on $\Rplus$ with
  $\Q_{M, N}[\DD{\mm{\nu}}] \leq \mu_0$ for all $N \in \N$,
  \item $ \lim\limits_{\delta \searrow 0} \limsup\limits_{N \to \infty}
  \Q_{M, N}[\mm{\nu}(\set{x \in X | \mm{\nu}(\clBall(x, \epsilon)) < \delta})]
  = 0$ for every $\epsilon>0$.
\end{enumerate}

\smallskip

\begin{enumerate}
  \item Define the finite measure
  \begin{equation}\label{eq:pr:th:nestedKingmanM2M_1}
    \mu_0 \asn \delta_0 + \ExpDistribution{\gamma_s} +
    \ExpDistribution{\gamma_g}  \ast \ExpDistribution{\gamma_s},
  \end{equation}
  where $\ExpDistribution{\gamma}$ denotes the exponential
  distribution with parameter $\gamma>0$ and $\mu\ast\eta$ denotes the
  convolution of two distributions $\mu$ and $\eta$. The law of
  $r(\vect{x}, \vect{y})$ is bounded from above by the finite measure
  $\mu_0$ for all $\vect{x}, \vect{y} \in \N^2$ (\cf
  \eqref{eq:NestedKingmanDistrOfDistances}). Since
  $\DD{\mm{\nu_{M, N}}} = r_*\mm{\nu_{M, N}}$, we get the desired inequality
  \begin{align*}
    \Q_{M, N}[\DD{\mm{\nu}}] = \int \DD{\mm{\nu_{M, N}}} \dif\P  \leq \mu_0.
  \end{align*}
  \item Let $\epsilon>0$. Then we have
  \begin{equation}\label{eq:pr:th:nestedKingmanM2M_2}
    \begin{aligned}
      \Q_{M, N}&[\mm{\nu}(\set{x \in X | \mm{\nu}(\clBall(x,
        \epsilon)) < \delta})] \\ &= \int \mm{\nu_{M, N}}(\set{x \in Z
        | \mm{\nu_{M,
            N}}(\clBall(x, \epsilon)) < \delta}) \dif\P \\
      &= \frac{1}{M} \sum_{i=1}^M \frac{1}{N} \sum_{j=1}^N\int
      \One_{[0, \delta)}\left(\mm{\nu_{M, N}}(\clBall((i,j), \epsilon))\right)
      \dif\P \\ &= \P(\mm{\nu_{M, N}}(\clBall((1,1), \epsilon)) <
      \delta)\\ &\leq \P(\mm{\nu_{M, N}}(\clBall((1,1), \epsilon))
      \leq \delta),
  \end{aligned}
\end{equation}
where we used the fact that
$\P(\mm{\nu_{M, N}}(\clBall((i,j), \epsilon)) < \delta)$ is the same
for all $i, j \in \N$. By the definition of the relative frequencies
$\freq{1}{i}{\epsilon}$ we have
  \begin{align*}
    \mm{\nu_{M,N}}(\clBall((1,1)), \epsilon)) \to  \frac{1}{M}
    \sum_{l=1}^M \freq{1}{l}{\epsilon} 
  \end{align*}
  almost surely for $N \to \infty$ and Fatou's lemma yields
  \begin{equation}\label{eq:pr:th:nestedKingmanM2M_3}
    \limsup_{N \to \infty}
    \P(\mm{\nu_{M,N}}(\clBall((1,1), \epsilon)) \leq \delta) \leq
    \P\left(\frac{1}{M}\sum_{l=1}^M\freq{1}{l}{\epsilon}
    \leq \delta\right).
  \end{equation}
  Combining \eqref{eq:pr:th:nestedKingmanM2M_2},
  \eqref{eq:pr:th:nestedKingmanM2M_3} and \eqref{eq:th:nestedKingmanM2M_0} we get
  \begin{align*}
    \lim_{\delta \searrow 0}&\limsup_{N \to \infty}\Q_{M,
    N}[\mm{\nu}(\set{x \in X | \mm{\nu}(\clBall(x, \epsilon)) < \delta})] \\
    &\leq \lim_{\delta \searrow 0}
      \P\left(\frac{1}{M}\sum_{l=1}^M\freq{1}{l}{\epsilon} 
    \leq \delta\right) \\ 
    &= \P\left(\frac{1}{M}\sum_{l=1}^M\freq{1}{l}{\epsilon} =0\right)\\
    &\leq \P(\freq{1}{1}{\epsilon}=0) \\ 
    &= 0 .
  \end{align*}
\end{enumerate}

\subsubsection{
Weak convergence of $(\Q_M)_M$}

We have no information about $\Q_M$ other than that it is the weak
limit of $(\Q_{M, N})_N$. Thus we must derive its properties from the
approximating measures $\Q_{M, N}$.

\paragraph{Uniqueness:}

We show that the sequence $(\Q_M)_M$ has at most one limit point by
proving existence of the limit $\lim_{M\to\infty}\Q_{M}[\Phi]$ for
each $\Phi$ as in~\eqref{eq:th:nestedKingmanM2M_1}. Because these
test functions are convergence determining, it follows that the
sequence $(\Q_M)_M$ has at most one limit point.

Fix a test function $\Phi$ of the
form~\eqref{eq:th:nestedKingmanM2M_1}. Since $\Phi$ is continuous and
bounded, we have
$\lim_{M\to\infty}\Q_M[\Phi] =
\lim_{M\to\infty}\lim_{N\to\infty}\Q_{M, N}[\Phi]$.
Using~\eqref{eq:th:nestedKingmanM2M_2} and the partial
exchangeability of the distances
(Remark~\ref{rm:NestedKingmanIsPartExchangeable}) one can show that
\begin{align*}
  \lim_{M\to\infty}\Q_M[\Phi] &=
  \lim_{M\to\infty}\lim_{N\to\infty}\Q_{M, N}[\Phi] \\
  &= \int
  \phi\big(R\big((1,1), \dotsc, (1, n_1), (2, 1), \dotsc, (m, n_m)\big)\big)\dif\P.
\end{align*}
Again, we omit the cumbersome proof. Heuristically, if $M \to \infty$,
then the probability to sample a species more than once goes to 0.

\paragraph{Relative compactness:}

Again, we use Corollary~\ref{cor:MtwoProbTightness2}. We will show the
following:
\begin{enumerate}
  \item $\Q_{M}[\DD{\mm{\nu}}] \leq \mu_0$ for every $M \in \N$,
  where $\mu_0$ is defined in~\eqref{eq:pr:th:nestedKingmanM2M_1},
  \item
  $\lim\limits_{\delta \searrow 0} \limsup\limits_{M \to \infty}
  \Q_{M}[\mm{\nu}(\set{x \in X | \mm{\nu}(\clBall(x, \epsilon)) <
    \delta})] = 0$ for every $\epsilon>0$.
\end{enumerate}

\smallskip

\begin{enumerate}
  \item Fix $M \in \N$. Observe that for finite measures
  $\eta, \mu \in \calM_f(\Rplus)$ we have $\eta \leq \mu$ if and only
  if $\int f \dif\eta \leq \int f \dif\mu$ for every non-negative,
  bounded and continuous function $f$. For such $f$ the function
  \begin{align*}
    \MtwoProb &\to \Rplus \\
    (X, r, \nu) &\mapsto \int f \dif\DD{\mm{\nu}}
  \end{align*}
  is bounded and continuous as a concatenation of bounded and
  continuous functions. Because $\Q_{M}$ is the weak limit of
  $(\Q_{M, N})_N$ and because
  $\int f \dif\Q_{M, N}[\DD{\mm{\nu}}] \leq \int f \dif\mu_0$, we get
  \begin{align*}
    \int f \dif\Q_M[\DD{\mm{\nu}}] &= \int \int f \dif\DD{\mm{\nu}}
                                     \dif\Q_M  \\
                                   &= \lim_{N \to \infty}\int \int f \dif\DD{\mm{\nu}}
                                     \dif\Q_{M, N} \\
                                   &= \lim_{N \to \infty}\int f \dif\Q_{M, N}[\DD{\mm{\nu}}]
                                    \\ &\leq \int f \dif\mu_0.
  \end{align*}
  Therefore, we have $\Q_M[\DD{\mm{\nu}}] \leq \mu_0$ for every
  $M \in \N$.
  \item Fix $\epsilon > 0$. Because $\Q_{M, N}$ converges weakly to
  $\Q_M$, we have
  \begin{align*}
  \liminf_{N \to \infty}\Q_{M, N}[f] \geq \Q_M[f]
  \end{align*}
  for every bounded lower semi-continuous function
  $f \from \Mtwo \to \R$ (\cf \cite[Corollary 8.2.5]{Bogachev}).
  By Lemma~\ref{lm:fKisContinuous} and Lemma~\ref{lm:AMMDLowerSemiCts}
  the function
  \begin{align*}
    \MtwoProb &\to \Rplus \\
    (X, r, \nu) &\mapsto \mm{\nu}(\set{x \in X | \mm{\nu}(\clBall(x,
    \epsilon) < \delta)})
  \end{align*}
  is lower semi-continuous (and obviously bounded by 1). Therefore, we
  have
  \begin{align*}
    &\Q_{M}[\mm{\nu}(\set{x \in X | \mm{\nu}(\clBall(x, \epsilon)) <
    \delta})] \\
    &\leq \liminf_{N \to \infty} \Q_{M, N}[\mm{\nu}(\set{x
    \in X | \mm{\nu}(\clBall(x, \epsilon)) < \delta})].
  \end{align*}
  Using inequalities \eqref{eq:pr:th:nestedKingmanM2M_2} and
  \eqref{eq:pr:th:nestedKingmanM2M_3} we get
  \begin{align}\label{eq:pr:th:nestedKingmanM2M_4}
    \Q_{M}[\mm{\nu}(\set{x \in X | \mm{\nu}(\clBall(x, \epsilon)) <
    \delta})] &\leq
                \P\left(\frac{1}{M}\sum_{l=1}^M\freq{1}{l}{\epsilon}
                \leq \delta\right). 
  \end{align}
  
  $(\freq{1}{l}{\epsilon})_{l \geq 2}$ is a sequence of exchangeable
  random variables. By de Finetti's Theorem (\cf\cite[Theorem
  3.1]{Aldous85}) there exists a random probability measure $\Xi$ with
  values in $\calM_1([0, 1])$ such that $\Xi^{\otimes \N}$ is a regular
  conditional distribution of $(\freq{1}{l}{\epsilon})_{l \geq 2}$ given
  $\sigma(\Xi)$. In other words, $(\freq{1}{l}{\epsilon})_{l \geq 2}$
  is conditionally \iid given $\sigma(\Xi)$. It follows that
  \begin{align*}
    \frac{1}{M}\sum_{l=2}^M\freq{1}{l}{\epsilon} \xrightarrow{M \to \infty}
    \E(\freq{1}{2}{\epsilon} \mid \Xi) = \int x \dif\Xi(x)
  \end{align*}
  almost surely (\cf \cite[Equation 2.24]{Aldous85}). Fatou's lemma
  yields
  \begin{equation}\label{eq:pr:th:nestedKingmanM2M_5}
    \begin{aligned}
      \lim_{\delta \searrow 0}\limsup_{M \to \infty}
      \P\left(\frac{1}{M}\sum_{l=1}^M\freq{1}{l}{\epsilon} \leq \delta\right)  &\leq
      \lim_{\delta \searrow 0} \P\left(\int x \dif\Xi(x) \leq \delta\right) \\
      &= \P\left(\int x \dif\Xi(x) = 0\right) \\
      &= \P(\Xi = \delta_0).
    \end{aligned}  
  \end{equation}
  Since $\Xi^{\otimes \N}$ is the conditional distribution of
  $(\freq{1}{l}{\epsilon})_{l \geq 2}$ given $\sigma(\Xi)$, we have
  \begin{align*}
    \P(\freq{1}{l}{\epsilon} = 0 \text{ for all }l \geq 2 \mid \Xi =
    \delta_0) = 1.
  \end{align*}
  It follows that
  \begin{equation}\label{eq:pr:th:nestedKingmanM2M_6}
    \begin{aligned}
      \P(\Xi = \delta_0) &= \P(\Xi = \delta_0)\cdot
      \P(\freq{1}{l}{\epsilon} = 0 \text{ for all }l \geq 2 \mid \Xi =
      \delta_0) \\ 
      &= \P\left(\freq{1}{l}{\epsilon} = 0 \text{ for all }l \geq 2\right).
    \end{aligned}
  \end{equation}
  The latter probability is 0 by Lemma~\ref{lm:AllFrequenciesAre0}. By
  combining~\eqref{eq:pr:th:nestedKingmanM2M_4},
  \eqref{eq:pr:th:nestedKingmanM2M_5}
  and~\eqref{eq:pr:th:nestedKingmanM2M_6} we get the claim.
\end{enumerate}

\paragraph{Acknowledgments:} The article originates from a PhD project
under the supervision of Anita Winter at the University of
Duisburg-Essen. I am very grateful to Anita Winter for her constant
support and encouragement during the last few years. I would also like
to thank Fabian Gerle, Volker Krätschmer and Wolfgang Löhr for several
helpful comments and discussions. Moreover, I am thankful to the DFG SPP
Priority Programme 1590 for financial support.

 \bibliography{literature}
 \bibliographystyle{amsalpha}

\end{document}